\documentclass[11pt]{amsart}
\usepackage{lmodern}
\usepackage{amsmath, amsthm, amssymb, amsfonts}
\usepackage[normalem]{ulem}
\usepackage{hyperref}

\usepackage{verbatim} 
\usepackage{longtable}

\usepackage{mathtools}

\usepackage{tikz}
\usetikzlibrary{decorations.pathmorphing}
\tikzset{snake it/.style={decorate, decoration=snake}}

\usepackage{caption}

\usepackage{tikz-cd}
\usetikzlibrary{arrows}

\theoremstyle{plain}
\newtheorem{thm}{Theorem}[section]
\newtheorem{cor}[thm]{Corollary}
\newtheorem{lem}[thm]{Lemma}
\newtheorem{prop}[thm]{Proposition}
\newtheorem{conj}[thm]{Conjecture}

\theoremstyle{definition}

\newtheorem{example}[thm]{Example}

\theoremstyle{remark}
\newtheorem{rmk}[thm]{Remark}

\newcommand{\BC}{{\mathbb{C}}}
\newcommand{\BD}{{\mathbb{D}}}

\newcommand{\BG}{{\mathbb{G}}}
\newcommand{\BH}{{\mathbb{H}}}

\newcommand{\BK}{{\mathbb{K}}}
\newcommand{\BL}{{\mathbb{L}}}

\newcommand{\BP}{{\mathbb{P}}}
\newcommand{\BQ}{{\mathbb{Q}}}
\newcommand{\BR}{{\mathbb{R}}}

\newcommand{\BZ}{{\mathbb{Z}}}

\newcommand{\CB}{{\mathcal B}}
\newcommand{\CC}{{\mathcal C}}

\newcommand{\CE}{{\mathcal E}}
\newcommand{\CF}{{\mathcal F}}
\newcommand{\CG}{{\mathcal G}}
\newcommand{\CH}{{\mathcal H}}
\newcommand{\CI}{{\mathcal I}}

\newcommand{\CL}{{\mathcal L}}
\newcommand{\CM}{{\mathcal M}}

\newcommand{\CO}{{\mathcal O}}
\newcommand{\CP}{{\mathcal P}}

\newcommand{\CS}{{\mathcal S}}
\newcommand{\CT}{{\mathcal T}}
\newcommand{\CU}{{\mathcal U}}

\DeclareFontFamily{OT1}{rsfs}{}
\DeclareFontShape{OT1}{rsfs}{n}{it}{<-> rsfs10}{}
\DeclareMathAlphabet{\curly}{OT1}{rsfs}{n}{it}

\newcommand{\Coh}{\mathrm{Coh}}

\usepackage{tikz}
\usepackage{lmodern}
\usetikzlibrary{decorations.pathmorphing}

\begin{document}
\title[Hitchin fibrations, abelian surfaces, and P=W]{Hitchin fibrations, abelian surfaces, \\ and the P=W conjecture}
\date{\today}

\author[M.A. de Cataldo]{Mark Andrea de~Cataldo} 
\address{Stony Brook University}
\email{mark.decataldo@stonybrook.edu}

\author[D. Maulik]{Davesh Maulik}
\address{Massachusetts Institute of Technology}
\email{maulik@mit.edu}

\author[J. Shen]{Junliang Shen}
\address{Yale University}
\email{junliang.shen@yale.edu}

\begin{abstract}
We study the topology of Hitchin fibrations via abelian surfaces. We establish the P=W conjecture for genus $2$ curves and arbitrary rank. In higher genus and arbitrary rank, we prove that P=W holds for the subalgebra of cohomology generated by even tautological classes. Furthermore, we show that all tautological generators lie in the correct pieces of the perverse filtration as predicted by the P=W conjecture. In combination with recent work of Mellit, this reduces the full conjecture to the multiplicativity of the perverse filtration. 

Our main technique is to study the Hitchin fibration as a degeneration of the Hilbert--Chow morphism associated with the moduli space of certain torsion sheaves on an abelian surface, where the symmetries induced by Markman's monodromy operators play a crucial role.
\end{abstract}

\baselineskip=14.5pt
\maketitle

\setcounter{tocdepth}{1} 

\tableofcontents
\setcounter{section}{-1}

\section{Introduction}

\subsection{Perverse filtrations}
 Throughout the paper, we work over the complex numbers $\BC$.
 
Let $\pi: X \rightarrow Y$ be a proper morphism with $X$ a nonsingular algebraic variety. The perverse $t$-structure on the constructible derived category $D_c^b(Y)$ with rational coefficients $\BQ$ induces a finite and  increasing filtration on the rational cohomology groups $H^*(X, \BQ)$
\begin{equation} \label{Perv_Filtration}
    P_0H^\ast(X, \BQ) \subset P_1H^\ast(X, \BQ) \subset \dots \subset P_kH^\ast(X, \BQ) \subset \dots \subset H^\ast(X, \BQ),
\end{equation}
called the \emph{perverse filtration} associated with $\pi$. The filtration (\ref{Perv_Filtration}) is governed by the topology of the morphism $\pi: X\rightarrow Y$. See Section \ref{sec1} for a brief review of the subject.

Every cohomology class on $X$ is indexed by the perverse filtration (\ref{Perv_Filtration}). We say that a class $\alpha \in H^d(X, \BQ)$ has \emph{perversity} $k$ if $\alpha = 0$, or
\[
\alpha \in P_kH^d(X, \BQ) \; \mathrm{and} \;  \alpha \notin P_{k-1}H^d(X, \BQ).
\]

The purpose of this paper is to study perverse filtrations associated with Hitchin fibrations in view of the \emph{P=W conjecture} by de~Cataldo, Hausel, and Migliorini \cite{dCHM1}. Our method is to use \emph{symmetries} induced by the monodromy of moduli spaces of sheaves on abelian surfaces.

\subsection{The P=W conjecture} 
Let $C$ be an irreducible  nonsingular  projective curve of genus $g \geq 2$. There are two moduli spaces which are attached to the curve $C$, the reductive Lie group $\mathrm{GL}_r$, and an integer $\chi$ with $\mathrm{gcd}(r,\chi)=1$.  
They are the twisted versions of Simpson's Dolbeault and Betti moduli spaces \cite{Si1994II}; see \cite{HT1} for  the non-abelian Hodge theory in the twisted  case.

The first moduli space $\CM_{\mathrm{Dol}}$ parametrizes stable Higgs bundles on $C$
\[
(\CE, \theta), \quad \theta: \CE \to \CE \otimes \Omega_C
\]
with $\mathrm{rank}(\CE)=r$ and $\chi(\CE) =\chi$. The variety $\CM_{\mathrm{Dol}}$ admits a projective morphism with connected fibers
\begin{equation}\label{Hitchin_fib}
h: \CM_{\mathrm{Dol}} \rightarrow \Lambda = \bigoplus_{i=1}^n H^0(C, \Omega_C^{\otimes i})
\end{equation}
sending $(\CE, \theta) \in \CM_{\mathrm{Dol}}$ to the characteristic polynomial $\mathrm{char}(\theta) \in \Lambda$. The proper morphism (\ref{Hitchin_fib}), called the \emph{Hitchin fibration}, is Lagrangian with respect to the canonical holomorphic symplectic form on~$\CM_{\mathrm{Dol}}$ given by the hyper-K\"ahler metric on $\CM_{\mathrm{Dol}}$; see \cite{Hit1}.  
The second moduli space is the  (twisted) character variety $\CM_B$. It  can be described (see \cite{HT1}, or  \cite{Shende} for a an alternative description) as parametrizing isomorphism classes of irreducible local systems
\[
\rho: \pi_1(C \setminus \{p\}) \to \mathrm{GL}_r
\]
where $\rho$ sends a loop around a chosen point $p$ to $\xi^\chi_r I_r \in \mathrm{GL}_r$, with $\xi_r$ a primitive root of unity. The character variety $\CM_B$ is affine. 

In \cite{Simp}, Simpson constructed a diffeomorphism between the (un-twisted) moduli spaces 
$\CM_{\mathrm{Dol}}$ and $\CM_B$, called the non-abelian Hodge theory;  see \cite{HT1} for the twisted case.  A striking 
prediction, suggested by the parallel between the Relative Hard Lefschetz \cite{BBD} and 
Curious Hard Lefschetz \cite{HRV} Theorems, was formulated by de~Cataldo, Hausel, and Migliorini \cite{dCHM1}; it predicts that the perverse filtration of $\CM_{\mathrm{Dol}}$ with respect to the Hitchin fibration (\ref{Hitchin_fib}) matches the weight filtration of the mixed Hodge structure on $\mathcal{M}_B$ under the identification
\[
H^\ast(\CM_{\mathrm{Dol}}, \BQ) = H^\ast(\CM_B, \BQ) 
\]
induced by Simpson's  (twisted) non-abelian Hodge theory.

\begin{conj}[\cite{dCHM1} P=W] \label{P=W_conj}
We have
\[
P_kH^\ast(\CM_{\mathrm{Dol}}, \BQ) = W_{2k}H^\ast(\CM_{B}, \BQ)= W_{2k+1}H^\ast(\CM_{B}, \BQ), \quad k \geq 0.
\]
\end{conj}

Conjecture \ref{P=W_conj} establishes a surprising connection between the topology of Hitchin fibrations and the Hodge theory of character varieties. It was proven in \cite{dCHM1} in the case of $r=2$ for any genus $g \geq 2$. See also \cite{dCHM3, SZ, Z} for certain parabolic cases, and \cite{SY, HLSY} for a compact analog concerning Lagrangian fibrations on projective holomorphic symplectic manifolds. 

The first main result of this paper is a proof of Conjecture \ref{P=W_conj} for curves of genus $2$ and arbitrary rank $r\geq 1$.

\begin{thm}\label{genus_2}
The P=W
Conjecture \ref{P=W_conj} holds when $C$ has genus $g=2$.
\end{thm}

For an irreducible nonsingular curve $C$ of genus $g \geq 2$, it is a general fact that $P=W$ for $\mathrm{GL}_r$ implies $P=W$ for $\mathrm{PGL}_r$; see (\ref{eqn71}) and (\ref{eqn89}). Hence we conclude immediately from Theorem \ref{genus_2} that the $P=W$ conjecture holds for $\mathrm{PGL}_r$ when the curve $C$ has genus $g=2$.

As explained in \cite[Section 1]{dCHM1}, the curious Poincar\'e duality and the Curious Hard Lefschetz conjectures \cite[Conjectures 4.2.4 and 4.2.7]{HRV} for character varieties are consequences of the P=W Conjecture \ref{P=W_conj}.  Moreover, by \cite{CDP} and \cite[Section 9.3]{MT}, the P=W Conjecture \ref{P=W_conj} implies the correspondence between Gopakumar--Vafa invariants and Pandharipande--Thomas invariants \cite[Conjecture 3.13]{MT} for the local Calabi--Yau 3-fold $T^*C \times \BC$ in the curve class $r[C]$. Hence Theorem \ref{genus_2} verifies all these conjectures for genus 2 curves.\footnote{As we discuss later, the curious Poincar\'e and Lefschetz conjectures have been recently shown by Mellit \cite{Mellit} to hold for all genera.}



\subsection{Tautological classes}\label{sec_0.3}
Assume that $C$ has genus $g \geq 2$ and $r \geq 1$. A set of generators for the  cohomology rings identified by the non-abelian Hodge diffeomorphism
\begin{equation}\label{cohomology}
H^\ast(\CM_{\mathrm{Dol}}, \BQ) = H^\ast(\CM_B, \BQ) 
\end{equation}
is described in \cite{Markman} by tautological classes. 

Let
\[
p_C: C \times \CM_{\mathrm{Dol}} \to C,\quad  p_\CM: C \times \CM_{\mathrm{Dol}} \to  \CM_{\mathrm{Dol}}\]
be the projections. We say that a triple
\[
(\CU^\alpha, \theta) = (\CU, \theta, \alpha)
\]
is a {\em twisted universal family} over $C \times \CM_{\mathrm{Dol}}$, if $(\CU, \theta)$ is a universal family and 
\[
\alpha = p_C^\ast \alpha_1 +p_\CM^\ast \alpha_2 \in H^2(C\times \CM_{\mathrm{Dol}}, \BQ),
\]
with $\alpha_1 \in H^2(C, \BQ)$ and $\alpha_2 \in H^2(\CM_{\mathrm{Dol}}, \BQ)$.

For a twisted universal family $(\CU^\alpha, \theta)$, we define the {\em twisted Chern character} $\mathrm{ch}^\alpha(\CU)$ as
\begin{equation*}\label{uni_eqn1}
\mathrm{ch}^\alpha(\CU) = \mathrm{ch}(\CU) \cup \mathrm{exp}(\alpha) \in H^*(C\times \CM_{\mathrm{Dol}}, \BQ),
\end{equation*}
and we denote by 
\[
\mathrm{ch}_k^\alpha(\CU) \in H^{2k}(C\times \CM_{\mathrm{Dol}},\BQ)
\]
its degree $k$ part. The class $\mathrm{ch}^\alpha(\CU)$ is called \emph{normalized} if 
\[
\mathrm{ch}^\alpha_1(\CU)|_{p \times \CM_{\mathrm{Dol}}} =0 \in H^2(\CM_{\mathrm{Dol}}, \BQ), \qquad \mathrm{ch}^\alpha_1(\CU)|_{C \times q} =0 \in H^2(C, \BQ),
\]
with $p \in \CM_{\mathrm{Dol}}$ and $q\in C$ points.  Since two universal families differ by the pull-back
of a line bundle on $\CM_{\mathrm{Dol}}$,  a straightforward calculation shows that  normalized classes on $C\times \CM_{\mathrm{Dol}}$  exist
and are uniquely determined. We introduce the tautological classes associated with the normalized class $\mathrm{ch}^\alpha(\CU)$ as follows.


For any $\gamma \in H^i(C, \BQ)$, let $c(\gamma ,k)$ denote the \emph{tautological class}
\begin{equation}\label{taut_class}
c(\gamma ,k) := \int_\gamma \mathrm{ch}_k^\alpha(\CU) = {p_{\CM}}_\ast( p_C^\ast \gamma \cup \mathrm{ch}^\alpha_k(\CU)) \in H^{i+2k-2}(\CM_{\mathrm{Dol}}, \BQ).
\end{equation}
Markman showed in \cite{Markman} that the classes $c(\gamma,k)$ generate the cohomology ring (\ref{cohomology}) as a $\BQ$-algebra. Furthermore, Shende proved in \cite{Shende} { (cf. Lemma \ref{lem3.1})}  that 
\begin{equation}\label{taut_weight}
    c(\gamma,k) \in {^k\mathrm{Hdg}^{i+2k-2}(\CM_B)},\quad \forall~\gamma \in H^i(C, \BQ),
\end{equation}
where 
\[
^k\mathrm{Hdg}^d(\CM_B) =  W_{2k}H^d(\CM_B, \BQ) \cap F^k H^d(\CM_B, \BC) \cap \bar{F}^k H^d(\CM_B, \BC),
\]
with $F^*$ and $W_*$ the Hodge  and weight filtrations, respectively. It follows that  the rational cohomology  $H^*(\CM_B, \BQ)$ is split  of Hodge--Tate type, and there is a canonical decomposition of 
the graded vector spaces (\ref{cohomology}) induced by the mixed  Hodge theory of $\CM_B$,
\begin{equation}\label{weight_decomp}
H^\ast(\CM_{\mathrm{Dol}}, \BQ) = H^\ast(\CM_B, \BQ) = \bigoplus_{k,d}  {^k\mathrm{Hdg}^d(\CM_B)}.
\end{equation}


The P=W conjecture, 
which can be restated as predicting that
\begin{equation}\label{p=w-explicit}
P_k H^\ast(\CM_{\mathrm{Dol}}, \BQ) = \bigoplus_{k' \leq k}  {^{k'}\mathrm{Hdg}^*(\CM_B)},
\end{equation} 
 is equivalent to the following statements.


\begin{conj}[Equivalent version of P=W]\label{P=W2} 
There exists a splitting $G_*H^*(\CM_{\mathrm{Dol}}, \BQ)$ of the perverse filtration associated with the Hitchin fibration (\ref{Hitchin_fib}) satisfying the following two properties.
\begin{enumerate}
    \item[(a)](Tautological classes) 
   Every tautological class $c(\gamma,k)$ has perversity $k$ and in fact
    \begin{equation*}\label{taut_conj}
    c(\gamma,k) \in G_kH^{\ast}(\CM_{\mathrm{Dol}}, \BQ),\quad \forall k \geq 0, \, \forall \gamma \in H^*(C, \BQ).
    \end{equation*}
  
    \item[(b)] (Multiplicativity) The perverse decomposition is multiplicative, \emph{i.e.}, 
    \[
    \cup: G_k H^d(\CM_\mathrm{Dol}, \BQ)\times G_{k'} H^{d'}(\CM_\mathrm{Dol}, \BQ) \to G_{k+k'} H^{d+d'}(\CM_\mathrm{Dol}, \BQ).
    \]
\end{enumerate}
\end{conj}

We recall that the multiplicativity of the weight filtration
\[
\cup: W_k H^d(X, \BQ) \times W_{k'} H^{d'}(X, \BQ) \to W_{k+k'}H^{d+d'}(X, \BQ)
\]
is standard from mixed Hodge theory. However, the perverse filtration associated with a proper flat morphism is not multiplicative in general; see \cite[Exercise 5.6.8]{Park}. It is mysterious why the perverse filtration associated with the Hitchin fibration (\ref{Hitchin_fib}) should be multiplicative, as is predicted by the P=W conjecture. In fact, one consequence of this paper is that the full P=W conjecture can be reduced to the multiplicativity of the perverse filtration; see Theorem \ref{last}.  This approach of analyzing the P=W conjecture via Conjecture \ref{P=W2} goes back to \cite{dCHM1} where it is applied to prove the case of rank $2$.

\subsection{P=W for tautological classes}
Theorem \ref{genus_2} is concerned with genus $g=2$.
In this section, we state our results for genus $g \geq  2$.
\subsubsection{Even tautological classes}
We consider the $\BQ$-subalgebra
\[
R^*(\CM_{\mathrm{Dol}})  \subset H^*(\CM_{\mathrm{Dol}}, \BQ)= H^*(\CM_B, \BQ)
\]
generated by all the \emph{even} tautological classes
\[
c(\gamma,k), \quad  k\geq 0, \quad  \gamma \in H^0(C, \BQ)\oplus H^2(C, \BQ).
\]
Our next theorem establishes the P=W conjecture in arbitrary genus and rank, but restricted to this subalgebra.
\begin{thm}\label{P=W_even}
The P=W conjecture holds, for any genus $g \geq 2$ and rank $r\geq 1$ for $R^*(\CM_{\mathrm{Dol}})$, \emph{i.e.}
\[
\begin{split}
P_kH^\ast(\CM_{\mathrm{Dol}}, \BQ) \cap R^*(\CM_{\mathrm{Dol}}) & = W_{2k}H^\ast(\CM_{B}, \BQ) \cap R^*(\CM_{\mathrm{Dol}})\\
& = W_{2k+1}H^\ast(\CM_{B}, \BQ) \cap R^*(\CM_{\mathrm{Dol}}).
\end{split}
\]
\end{thm}

In fact, in the proof of Theorem  \ref{P=W_even}, we obtain a multiplicative decomposition 
\begin{equation*}
R^*(\CM_{\mathrm{Dol}}) = \bigoplus_{k,d} G_k R^d(\CM_{\mathrm{Dol}})
\end{equation*}
 that splits the restricted perverse filtration $P_kH^\ast(\CM_{\mathrm{Dol}}, \BQ) \cap R^*(\CM_{\mathrm{Dol}})$ and such that 
\[
G_k R^d(\CM_{\mathrm{Dol}}) =  {^k\mathrm{Hdg}^d}(\CM_B) \cap R^{d}(\CM_{\mathrm{Dol}}).
\]
We refer to Section \ref{Section4.6} for more details.

\subsubsection{Odd tautological classes}
The following theorem concerns odd tautological classes
\begin{equation}\label{odd_taut}
c(\gamma, k) \in H^{2k-1}(\CM_{\mathrm{Dol}}, \BQ), \quad \forall k \geq 0, \, \forall \gamma \in H^1(C, \BQ).
\end{equation}

\begin{thm}\label{P=W_odd}
Any odd tautological class (\ref{odd_taut}) has perversity $k$.
\end{thm}

Theorems \ref{P=W_even} and \ref{P=W_odd} provide strong evidence for the P=W Conjecture
\ref{P=W_conj} (and for its reformulation Conjecture \ref{P=W2}) for any genus $g \geq 2$. In particular, we obtain that all tautological generators $c(\gamma,k)$ satisfy the $P=W$ match:
\begin{equation}\label{eqn9}
    c(\gamma, k) \in P_k H^*(\CM_{\mathrm{Dol}}, \BQ),\quad \forall~ \gamma \in H^*(C, \BQ).
\end{equation}

\subsubsection{Multiplicativity is equivalent to P=W} The P=W Conjecture \ref{P=W_conj}  implies immediately that the perverse filtration $P_*H^*(\CM_{\mathrm{Dol}}, \BQ)$ is multiplicative, \emph{i.e.}
\begin{equation}\label{multi1}
\cup: P_k H^d(\CM_\mathrm{Dol}, \BQ)\times P_{k'} H^{d'}(\CM_\mathrm{Dol}, \BQ) \to P_{k+k'} H^{d+d'}(\CM_\mathrm{Dol}, \BQ).
\end{equation}
The following theorem shows that the converse is also true. It is a corollary of Theorems \ref{P=W_even} and \ref{P=W_odd}, and of the Curious Hard Lefschetz conjecture recently established by Mellit \cite{Mellit}.

\begin{thm}\label{last}
The P=W Conjecture \ref{P=W_conj} is equivalent to the multiplicativity (\ref{multi1}) of the perverse filtration.
\end{thm}

In fact, we deduce from (\ref{taut_weight}), (\ref{eqn9}), and (\ref{multi1}) that
\[
W_{2i}H^*(\CM_B, \BQ) \subset P_iH^*(\CM_{\mathrm{Dol}}, \BQ), \quad \forall i\geq 0,
\]
which further implies Conjecture \ref{P=W_conj} by \cite[Theorem 1.5.3]{Mellit}
and the second paragraph following \cite[Corollary 1.5.4]{Mellit}.

{
\begin{rmk}\label{r=1}
All the results of this paper hold for arbitrary rank $r \geq 1$. The case $r=1$ is classical and easy: the Dolbeault moduli space is the cotangent bundle of a Jacobian of some degree  of the curve and the Betti moduli space
is a product of $\mathbb G_m^{2g}$. In this paper, starting with the assumption {$\beta^2\geq 6$} in Section \ref{Sec2.1}, we focus on the case of rank $r\geq 2$, {which corresponds to $\beta^2 \geq 8$}. The case of rank one could be dealt with also by using the methods of this paper, but at the price of some easy modifications that we deemed a distraction. 
\end{rmk}
}
\subsection{Idea of the proof}

The key idea in proving our results
is to first study perverse filtrations for certain \emph{compact} hyperk\"ahler manifolds, motivated in part by earlier work of the third author and Yin \cite{SY}.
That is, we consider an abelian surface $A$ with ample curve class $\beta \in H_2(A, \BZ)$, and prove analogs of 
our main result for the moduli space
$\CM_{\beta,A}$ of one-dimensional sheaves with support in the class $\beta$; see Theorem \ref{taut_abelian} for the precise statement.
The advantage of the compact geometry is that we can use
results of Markman \cite{Markman3} on monodromy operators for $\CM_{\beta,A}$.  These are symmetries, arising from parallel transport and Fourier-Mukai transforms, that do not appear in the Hitchin setting and which heavily constrain tautological classes of universal families. We show that these operators relate perverse filtrations and tautological classes for different choices of $\beta$ and, in particular, allow us to pass from the case of imprimitive $\beta$ to primitive $\beta$, which can be studied directly using \cite{Z,SZ}.

In order to apply this to Hitchin fibrations, we consider the degeneration to the normal cone of an embedding
of {a genus $g$ curve $C$} into an abelian surface
\[
j_C: C\hookrightarrow A,
\]
and study the specialization map on cohomology of the associated moduli spaces.  In general, there is a great deal of information loss in this specialization because the family is non-proper. Nevertheless, we are able to show its compatibility with tautological classes and perverse splittings.  Finally we use these compatibilities  to deduce our main theorems {for such a curve $C$; this implies our results hold  for all curves in view of \cite{dCM}.}

\subsection{Outline of paper}

We briefly outline the contents of this paper.  In Section $1$, we recall some basic facts about perverse filtrations and earlier results of \cite{Z,SZ} for Hilbert schemes of points on the special abelian surface $E \times E'$.
In Section $2$, we pass to the setting of abelian surfaces.  One  key result  here is the splitting Theorem \ref{taut_abelian}.  After recalling Markman's work on monodromy operators, we prove the result for primitive classes $\beta$ by a direct analysis, and then we combine the primitive case with Markman's results  to establish the general case.
In Section $3$, we formulate and prove  Theorem \ref{strengthened2.1},  which is a strengthened version of the splitting of Theorem \ref{taut_abelian} and which  is  more robust for specialization arguments.
Finally, in section $4$, we study  specialization maps for our degeneration, and we use them to prove the main theorems.

While this paper is written in logical order, some readers may prefer to start from \S\ref{se4}, and refer backwards as needed.

\subsection{Acknowledgements}
We are grateful to Tam\'as Hausel, Jochen Heinloth, Luca Migliorini, Vivek Shende, David Vogan, Qizheng Yin, Zhiwei Yun, and Zili Zhang for helpful discussions.  M.A. de Cataldo is partially supported by NSF grants DMS 1600515 and 1901975.  D. Maulik is partially supported by NSF FRG grant DMS-1159265. J. Shen is supported by the NSF grant DMS 2134315.

\section{Perverse filtrations} \label{sec1}

\subsection{Overview}
We begin with the definition and some relevant properties of the perverse filtration associated with a proper surjective morphism $\pi: X\rightarrow Y$. Some  references are \cite{BBD, dCM0, dCM1, WSB}. 
Throughout this section, for simplicity, we assume $X$ and $Y$ are nonsingular.

Following \cite{dCHM3, Z, SZ}, we discuss the perverse filtration for the Hilbert scheme of $n$-points on an abelian surface $E \times E'$ product of two elliptic curves, induced by the natural  second projection morphism
\[
E \times E' \to E'.
\]
 This example plays a crucial role in Section 2 concerning moduli of sheaves on abelian surfaces.

\subsection{Perverse sheaves}
Let $D_c^b(Y)$ denote the bounded derived category of $\BQ$-constructible sheaves on $Y$, and let $\BD: D_c^b(Y)^{\mathrm{op}} \to D_c^b(Y)$ be the Verdier duality functor. The full subcategories 
\begin{align*}
    ^ \mathfrak{p} D_{\leq 0}^b(Y) & = \left\{\CE \in D_c^b(Y): \dim \mathrm{Supp}(\CH^i(\CE)) \leq -i \right\}, \\
    ^ \mathfrak{p} D_{\geq 0}^b(Y) & = \left\{\CE \in D_c^b(Y): \dim \mathrm{Supp}(\CH^i(\BD\CE)) \leq -i \right\}
\end{align*}
give rise to the \emph{perverse $t$-structure} on $D_c^b(Y)$, whose heart is the abelian category of perverse sheaves,
\[
\mathrm{Perv}(Y) \subset D_c^b(Y).
\]

For $k \in \BZ$, let $^\mathfrak{p}\tau_{\leq k}$ be the truncation functor associated with the perverse $t$-structure. Given an object $\CC \in D_c^b(Y)$, there is a natural morphism
\begin{equation}\label{trun0}
^\mathfrak{p}\tau_{\le k}\CC \rightarrow \CC.
\end{equation}
For the morphism $\pi:X \to Y$, we obtain from (\ref{trun0}) the morphism
\[
^\mathfrak{p}\tau_{\le k}R\pi_\ast \BQ_X \rightarrow R\pi_\ast \BQ_X,
\]
which further induces a morphism of (hyper-)cohomology groups
\begin{equation}\label{perv_filt}
H^{d-(\dim X -{ R})}\Big{(}Y, \, ^\mathfrak{p}\tau_{\le k}(R\pi_\ast \BQ_X[\dim X - { R}]) \Big{)} \rightarrow H^d(X, \BQ).
\end{equation}
Here
\begin{equation}\label{defect}
    {R} = \dim X \times_Y X - \dim X
\end{equation}
is the \emph{defect of semismallness}. The $k$-th piece of the perverse filtration (\ref{Perv_Filtration})
\[
P_kH^d(X, \BQ) \subset H^d(X, \BQ)
\]
is defined to be the image of (\ref{perv_filt}).\footnote{Here the shift $[\dim X- {R}]$ is to ensure that the perverse filtration is concentrated in the degrees $[0, 2 { R}]$.} 

We say that a  graded vector space decomposition
\[
H^*(X, \BQ) = \bigoplus_{k,d} G_k H^d(X, \BQ)
\]
\emph{splits} the perverse filtration, if 
\[
P_k H^d(X, \BQ) = \bigoplus_{i \leq k} G_i H^d(X, \BQ), \quad \forall k,d.
\]

\subsection{The decomposition theorem}\label{sec1.3}
Perverse filtrations can be described through the decomposition theorem \cite{BBD, dCM0}. By applying the decomposition theorem to the morphism $\pi: X \to Y$, we obtain an isomorphism
\begin{equation}\label{decomp}
R\pi_*\BQ_X[\dim X- { R}]\simeq \bigoplus_{i=0}^{2{ R}}\mathcal{P}_i[-i] \in D^b_c(Y)
\end{equation}
with $\mathcal{P}_i  \in \mathrm{Perv}(Y)$. The perverse filtration can be identified as
\[
P_kH^d(X,\BQ)=\mathrm{Im}\Big\{H^{d-(\dim X - { R})}(Y, \bigoplus_{i=0}^k\mathcal{P}_i[-i])\to H^d(X,\BQ)\Big\}.
\]
In general, the isomorphism (\ref{decomp}) in the decomposition theorem is not canonical. Once we fix such an isomorphism $\phi$, we obtain a splitting $G_\ast H^\ast(X, \BQ)$ of the perverse filtration,
\[
P_kH^d(X, \BQ) = \bigoplus_{i\leq k} G_iH^d(X, \BQ).
\]
Here
\[
G_iH^d(X, \BQ) = \mathrm{Im}\Big\{ H^{d-(\dim X - { R})}(Y, \mathcal{P}_i[-i])\to H^d(X,\BQ)\Big\}
\]
with
\[
H^{d-(\dim X - {R})}(Y, \mathcal{P}_i[-i])\to H^d(X,\BQ)
\]
the morphism induced by $\phi^{-1}$.

\subsection{Perverse filtration for projective bases}
We review another description of the perverse filtration associated with
\[
\pi: X \to Y,
\]
when $Y$ is projective.

We fix $\eta$ to be an ample class on $Y$, and we consider
\[
L = \pi^* \eta \in H^2(X, \BQ).
\]
The class $L$ acts on $H^*(X, \BQ)$ as an nilpotent operator via cup product
\[
L: H^*(X, \BQ) \xrightarrow{\cup L} H^*(X, \BQ).
\]

The following proposition shows that the filtration (\ref{Perv_Filtration}) is completely described by an ample class on the base.  It is stated in the special case needed in this paper.

\begin{prop}[\cite{dCM0} Proposition 5.2.4.] \label{prop1.1}
Assume that
\begin{equation*}
\dim X = 2 \dim Y = 2R.
\end{equation*}
Then we have
\[
P_kH^m(X, \BQ) = \sum_{i\geq 1} \left(
\mathrm{Ker}(L^{R+k+i-m}) \cap \mathrm{Im}(L^{i-1}) \right) \cap H^m(X, \BQ).
\]
\end{prop}

\subsection{Abelian surfaces and Hilbert schemes}\label{Section1.5}
Let $A = E \times E'$ be an abelian surface with $E$ and $E'$ elliptic curves, and let $n \in \BZ^{\geq 0}$.
The natural projection, 
\[
p: A \to E'
\]
induces a morphism
\[
p_n: A^{[n]} \to {E'}^{(n)}
\]
from the Hilbert scheme of $n$ points on $A$ to the symmetric product of the elliptic curve $E'$. 
We briefly review the construction \cite{SZ} of a  \emph{canonical} splitting 
\begin{equation}\label{canonical_split}
H^\ast(A^{[n]}, \BQ) = \bigoplus_{i,d}\widetilde{G}_iH^d(A^{[n]},\BQ)
\end{equation}
of the perverse filtration associated with $p_n$, which culminates with  the explicit formula (\ref{cansplitting15}).

We start with a canonical decomposition 
\begin{equation}\label{split1}
    H^*(A, \BQ) = \bigoplus_{i,d} \widetilde{G}_iH^d(A, \BQ)
\end{equation}
with $\widetilde{G}_iH^d(A, \BQ)$ the K\"unneth factor 
\[
H^i(E, \BQ) \otimes H^{d-i}(E',\BQ) \subset H^d(A, \BQ).
\]
Since $p: A\to E'$ is a trivial fibration, (\ref{split1}) splits the perverse filtration associated with $p$. Assume that we already have decompositions of $H^*(X_1, \BQ)$ and $H^*(X_2, \BQ)$, a decomposition of $H^\ast(X_1 \times X_2, \BQ)$ can be constructed by using the K\"unneth decomposition. In particular, we obtain the direct sum decomposition of $H^*(A^n, \BQ)$ with summands
\[
\widetilde{G}_kH^*(A^n, \BQ)= \left \langle \alpha_1 \boxtimes \cdots \boxtimes \alpha_n;~~ \alpha \in \widetilde{G}_{k_i}H^*(A, \BQ),~~ \sum_{i} {k_i} = k   \right \rangle,
\]
whose $\mathfrak{S}_n$-invariant part induces a decomposition 
\[
H^*(A^{(n)}, \BQ) = \bigoplus \widetilde{G}_iH^d(A^{(n)}, \BQ).
\]
In turn, by using  K\"unneth decompositions again, this gives us  canonical decompositions for
\[
H^\ast(A^{(n_1)}\times A^{(n_2)} \times \cdots \times A^{(n_k)}, \BQ), \quad n_i \geq 1.
\]
Finally, the cohomology of the Hilbert scheme $A^{[n]}$ is related to the cohomology of symmetric products by \cite{Go2, dCM3}.  We employ exponential notation for partitions $\nu:= (\nu_i)_{i=1}^l$, with  $\nu_i >0$ and
$\sum_i \nu_i =n,$ of a positive integer $n$. Namely, we let $a_k$ be the number of times the integer $k$ appears
in the partition, so that the expression 
\[
1^{a_1}2^{a_2}\cdots n^{a_n}\]
 recovers the partition $\nu$.
We use $A^{(\nu)}$ to denote the variety $A^{(a_1)} \times A^{(a_2)} \times \cdots \times A^{(a_n)}$. The cohomology group $H^d(A^{[n]}, \BQ)$ admits a canonical decomposition
\begin{equation}\label{GS_decomp}
H^d(A^{[n]}, \BQ) = \bigoplus_{\nu} H^{d+2l(\nu)-2n}(A^{(\nu)}, \BQ)
\end{equation}
where $\nu$ runs through all partitions of $n$ and $l(\nu)$ is the length of $\nu$. 
The desired canonical splitting (\ref{canonical_split}) is then defined by setting
\begin{equation}\label{cansplitting15}
\widetilde{G}_kH^d(A^{[n]}, \BQ) = \bigoplus_{\nu} \widetilde{G}_{k+l(\nu)-n}H^{d+2l(\nu)-2n}(A^{(\nu)}, \BQ)
\end{equation}
to be the sub-vector space of $H^d(A^{[n]}, \BQ)$ under the identification (\ref{GS_decomp}). By \cite[Proposition 4.12]{Z}, this decomposition splits the perverse filtration associated with $p_n$.\footnote{We will give another interpretation of the splitting $\tilde{G}_*H^*(A^{[n]},\BQ)$ in Corollary \ref{cor3.6}.}

 In view of \cite[Proposition 2.1]{Z},  by using the K\"unneth decomposition again, we obtain a canonical splitting 
\begin{equation}\label{splitting0}
H^\ast(A\times A^{[n]}, \BQ) = \bigoplus_{i,d} \widetilde{G}_i H^d(A \times A^{[n]}, \BQ)
\end{equation}
of the perverse filtration associated with
\[
p \times p_n: A \times A^{[n]} \to E' \times E'^{(n)}.
\]

The following theorem obtained in \cite{Z, SZ} is concerned with  tautological classes and  multiplicativity in the context of Hilbert schemes. 

\begin{thm}\label{taut_Hilb}
Let $\CI_n$ be the universal ideal sheaf on $A \times A^{[n]}$. 
\begin{enumerate}
    \item[(a)](Tautological classes) We have
    \[
    \mathrm{ch}_k(\CI_n) \in \widetilde{G}_k H^{2k}(A\times A^{[n]}, \BQ).
    \]
    \item[(b)](Multiplicativity) The decomposition  (\ref{canonical_split}) is multiplicative, i.e.
    \[
   \cup:  \widetilde{G}_i H^d(A^{[n]}, \BQ) \times \widetilde{G}_{i'} H^{d'}(A^{[n]}, \BQ) \to \widetilde{G}_{i+i'}H^{d+d'}(A^{[n]}, \BQ).
    \]
\end{enumerate}
\end{thm}

\begin{proof}
Theorem \ref{taut_Hilb} (b) was proven in \cite{Z}. Although the multiplicativity was shown only for the perverse filtration in \cite{Z}, the same proof works for  the decomposition (\ref{canonical_split}).  This was explained in \cite[Proposition 1.8]{SZ}. Theorem \ref{taut_Hilb} (a) follows directly from the proof of \cite[Theorem 0.1]{SZ} and the proof of \cite[Theorem 3.3]{SZ}, where it is explained how to treat perverse decompositions instead of perverse filtrations. 
\end{proof}


\section{Moduli of one-dimensional sheaves}
\subsection{Overview and main results}\label{Sec2.1}
Throughout this section, we assume $A$ is an abelian surface and fix  $\chi \in \mathbb{Z}$.  Let $\beta \in H_2(A, \BZ)$ be an ample curve class with $\beta^2 \geq 6$; we only consider classes $\beta$ for which the vector $(0,\beta,\chi)\in H^{\mathrm{ev}}(A,\BZ)$ is primitive.
Finally, we assume that we are given a polarization $H$ that is generic with respect to the vector $(0, \beta, \chi)$.  We will typically omit $\chi$ and $H$ from our notation.

Let $\CM_{\beta,A}$ be the moduli space which parametrizes one-dimensional sheaves satisfying
\begin{equation}\label{given_d}
[\mathrm{supp}(\CF)]= \beta, \quad \chi(\CF) = \chi,
\end{equation}
that are Gieseker-stable with respect to our generic polarization. In the paper, by ``support", we mean Fitting support (see \cite{LP1}), and the square bracket denotes the associated homology class.
Equivalently, these are Gieseker-stable sheaves on $A$ with Mukai vector $(0, \beta, \chi)$.
By  \cite{Yo}, the moduli space $\CM_{\beta,A}$
 is a non-empty, nonsingular projective variety of dimension $\beta^2+2$.

Let
\[
\pi_A: A\times \CM_{\beta,A} \to A, \quad \pi_\CM: A\times \CM_{\beta,A} \to \CM_{\beta,A}
\]
be the projections. 
If there exists a universal family $\CF_\beta$ on $A \times \CM_{\beta,A}$, then, in analogy to Section \ref{sec_0.3}, we define the twisted Chern character associated with a class 
\begin{equation}\label{alpha}
\alpha = \pi_A^\ast \alpha_1 + \pi_\CM^\ast \alpha_2 \in H^2(A\times \CM_{\beta,A}, \BC)
\end{equation}
by setting
\[
\mathrm{ch}^\alpha(\CF_\beta) = \mathrm{ch}(\CF_\beta) \cup \mathrm{exp}(\alpha) \in H^{*}(A \times \CM_{\beta,A}, \BC)
\]
(the use of $\BC$-coefficients is for later use in Theorem \ref{taut_abelian}),
and we denote its degree  $2k$ part by
\[
\mathrm{ch}_k^\alpha(\CF_\beta) \in H^{2k}(A\times \CM_{\beta,A}, \BC).
\]
In general, a universal family may not exist; however, as explained in \cite[Section 3.1]{Markman5}, there exists
a \emph{universal class} 
\[
[\CF_\beta] \in K_{\mathrm{top}}(A \times \CM_{\beta,A}),
\]
which is well-defined up to pullback of a topological line bundle on $\CM_{\beta,A}$.\footnote{Here $K_{\mathrm{top}}(X)$ denotes the Grothendieck ring of topological complex vector bundles on $X$ \cite{At}.}
As before, given a choice of $\alpha$ as in (\ref{alpha}), we can take its associated twisted Chern character 
\[
\mathrm{ch}^{\alpha}([\CF_\beta]) \in H^{*}(A \times \CM_{\beta,A}, \BC),
\]
which we still denote by $\mathrm{ch}^{\alpha}(\CF_\beta)$ for notational convenience.

The moduli space $\CM_{\beta,A}$ admits a natural Hilbert--Chow morphism
\begin{equation}\label{Hilb-Ch}
\pi_\beta: \CM_{\beta,A} \to B, \quad \pi_\beta(\CF) = [\mathrm{supp}(\CF)]
\end{equation}
where $B$ is  the Chow variety parametrizing effective one-cycles in the class $\beta$; see \cite{LP1, LP2}. The purpose of Section 2 is to prove the following theorem, which can be viewed as an analog of Conjecture \ref{P=W2} for abelian surfaces.

\begin{thm}\label{taut_abelian}
There exists a choice of universal twisted class
$\mathrm{ch}^\alpha(\CF_\beta)$ with $\alpha$ of  type (\ref{alpha}),
and a splitting $\widetilde{G}_\ast H^\ast(\CM_{\beta,A}, \BC)$ of the perverse filtration associated with
the Hilbert-Chow morphism $\pi_\beta$ (\ref{Hilb-Ch}), satisfying the following properties.
\begin{enumerate}
    \item[(a)](Tautological classes) For any $\gamma \in H^*(A, \BC)$, we have
    \[
   \int_\gamma \mathrm{ch}^\alpha_k(\CF_\beta) \in \widetilde{G}_{k-1}H^*(\CM_{\beta,A}, \BC).
    \]
    \item[(b)](Multiplicativity) We have
    \[
   \cup: \widetilde{G}_i H^d(\CM_{\beta,A}, \BC) \times \widetilde{G}_{i'} H^{d'}(\CM_{\beta,A}, \BC) \to \widetilde{G}_{i+i'}H^{d+d'}(\CM_{\beta,A}, \BC).
    \]
\end{enumerate}
\end{thm}

Theorem \ref{taut_abelian} will be deduced from Theorem \ref{taut_Hilb}  and  Markman's results on the monodromy of holomorphic symplectic varieties
\cite{Markman2, Markman3}.
The strategy of the argument is to use Theorem \ref{taut_Hilb} and direct computations to study the case of primitive curve classes for special abelian varieties. To pass to imprimitive curve classes (which form a distinct deformation class of Lagrangian fibrations), we require a mild extension of Markman's work to the setting of complex coefficients.

In the next three sections, we review Markman's work, and in the subsequent three sections, we  prove 
Theorem \ref{taut_abelian}.

\subsection{Mukai lattices}

Let 
\[
H^\ast(A, \BZ) = S_A^+ \oplus S_A^-
\]
be the even/odd decomposition of the cohomology $H^\ast(A, \BZ)$ with
\[
S_A^+ = H^0(A, \BZ) \oplus H^2(A, \BZ)\oplus H^4(A,\BZ), \quad S_A^- = H^1(A,\BZ)\oplus H^3(A,\BZ).
\]
The Mukai pairing on $S_A^+$ is given by
\begin{equation}\label{Mukai_pair}
\langle a, b \rangle = \int_A {(a_1\cup b_1 - a_0\cup b_2-a_2\cup b_0)},
\end{equation}
where $a=(a_0, a_1,a_2)$ and $b=(b_0, b_1, b_2)$ are elements in $S_A^+$. We have
\[
(S_A^+, \langle ~~~, ~~~ \rangle) \cong U^{\oplus 4}
\]
with $U$ the parabolic plane. For a coherent sheaf $\CF$ on $A$, its Mukai vector 
\[
v(\CF) = \mathrm{ch}(\CF) \cup \sqrt{\mathrm{td}_A}  \in S_A^+
\]
coincides with the Chern character $\mathrm{ch}(\CF)$.

Throughout Section 2, we assume $v \in S_A^+$ is a primitive vector such that the moduli space $M_{v,A}$ of Gieseker-stable sheaves $\CF$ on $A$ (with respect to a very general polarization $H$) with Mukai vector $v(\CF) = v$ is non-empty. By \cite[Theorem 0.1]{Yo}, the variety $M_{v,A}$ is nonsingular and projective of dimension
\[
\mathrm{dim} (M_{v,A}) = \langle v,v\rangle +2,
\]
and it is deformation equivalent to 
\[
M_{(1,0,-n), A} = A^{[n]} \times \hat{A}
\]
where $\hat{A}$ is the dual abelian variety of $A$ and {$2n = \langle v ,v \rangle$.} 

A universal family on 
\[
A \times M_{(1,0,-n), A} = A \times A^{[n]} \times  \hat{A}
\]
is 
\begin{equation}\label{univ_ideal}
\CE_n = \pi_{12}^\ast \CI_n \otimes \pi_{13}^\ast \CP,
\end{equation}
with $\CI_n$ the universal ideal sheaf, $\CP$ the normalized Poincar\'e line bundle on $A \times \hat{A}$, and $\pi_{ij}$ the corresponding projections. 

\subsection{$\mathrm{Spin}$ representations}\label{spin8reps} 

The even cohomology $S_A^+$ together with the Mukai pairing is a unimodular lattice.  
Consider the corresponding Spin group for this lattice, denoted by $\mathrm{Spin}(S_A^+)$. We recall certain representations of this group involved in Markman's work following \cite[Section 3]{Markman3}.

Let $H^*(A,\BZ)$ denote the total cohomology of $A$ equipped with the pairing
\begin{equation}\label{pairing}
(a,b)_A = \int_{A}\tau(a)\cup b,
\end{equation}
where $\tau$ acts on $H^i(A,\BZ)$ by $(-1)^{i(i-1)/2}$.
The group $\mathrm{Spin}(S_A^+)$ admits an action by  parity-preserving  isometries on $H^*(A, \BZ)$
\[\mathrm{Spin}(S_A^+) \rightarrow \mathrm{Aut}(H^*(A, \BZ)),\]
which extends the natural representation on $S_A^+$.

We consider $V$ to be the lattice 
\[
V = H^1(A, \BZ) \oplus H^1(\hat{A} ,\BZ)
\]
carrying a symmetric bilinear form via the canonical identification \[
H^1(\hat{A}, \BZ) = H^1(A, \BZ)^*.
\]
There exists a representation 
\[\mathrm{Spin}(S_A^+) \rightarrow \mathrm{Aut}(V)\]
by isometries which we can extend to an action of $\mathrm{Spin}(S_A^+)$
on $\bigwedge^{k}V = H^{k}(A \times \hat{A},\BZ)$ for each $k$. This action will play a role in the proof
of Theorem \ref{thm2.6}.

These constructions are linear-algebraic and make sense for the groups $\mathrm{Spin}(S_A^+\otimes \BK)$ with general coefficients 
\[
\BK = \BZ, \BQ, \BR, \mathrm{or}~~ \BC.
\]

\subsection{Monodromy representations}\label{Monodromy} 
In this section, we recall Markman's monodromy operators \cite{Markman3} for abelian surfaces. They play a crucial role in the proof of Theorem \ref{taut_abelian}. We refer to \cite{Markman2} for a parallel theory concerning $K3$ surfaces. 

Let $H^\ast(A)$ be the cohomology $H^\ast(A, \BK)$ with coefficients $\BK = \BZ~~\mathrm{or}~~ \BC$.  Assume that
\[
g: H^*(A) \to H^*(A)
\]
is a parity-preserving homomorphism induced by an element (cf. Section \ref{spin8reps})
\[
g \in \mathrm{Spin}(S_A^+\otimes \BK)
\]
 which fixes a primitive Mukai vector $v\in S_A^+$. 
Assume we are given a generic polarization $H$ with respect to $v$.
We explain how to construct a graded $\BK$-algebra automorphism
\begin{equation*}\label{mono1}
\gamma_{g,v}:  H^*(M_{v,A}, \BK) \to H^*(M_{v,A}, \BK).
\end{equation*}


For a projective variety $X$, we define
\[
l: \bigoplus_i H^{2i}(X, \BQ) \rightarrow \bigoplus_i H^{2i}(X, \BQ) 
\]
to be the universal polynomial map which takes the exponential Chern character to its total Chern class,
\[
l(r+a_1+a_2 + \cdots) = 1+a_1 + (a_1^2/2-a_2) + \cdots,
\]
and we define
\[
(-)^\vee: \bigoplus_iH^\mathrm{2i}(X, \BZ) \to \bigoplus_i H^{2i}(X, \BZ), \quad \omega \mapsto \omega^\vee
\]
to be the dualizing automorphism which acts by $(-1)^i$  on the even cohomology $H^{2i}(X, \BZ)$. 

Let $[\CE_v]$ and $[\CE'_v]$ be universal classes on $A \times M_{v,A}$, in the sense of Section 2.1. We denote by $x_{g,v}$ the class  
\[
x_{g,v} =  \pi_{12}^\ast \left[(\mathrm{id}\otimes g)\mathrm{ch}(\CE_v)\right]^\vee \cup \pi_{23}^\ast \mathrm{ch}(\CE'_v) \in H^*(M_{v,A} \times A \times M_{v,A} ,\BK),
\]
where $\pi_{ij}$ are the projections from the product $M_{v,A}\times A \times M_{v,A}$ to the corresponding factors. 
Via the formalism of correspondences, to define an endomorphism of cohomology, it suffices to give a cohomology class in the product; in our case, 
the morphism $\gamma_{g,v}$ is defined by the class ()
\begin{equation}\label{gammagv}
 \gamma_{g,v}([\CE_v],[\CE'_v])=\left( \left[l({\pi_{13}}_\ast x_{g,v}) \right]^{-1} \right)_{\mathrm{deg}~2d_M} \in H^{2d_M}(M_{v,A} \times M_{v,A}, \BK)
\end{equation}
where $d_M = \mathrm{dim} (M_{v,A})$, and $-_{\mathrm{deg}~~ k}$ denotes the degree $k$ part of a cohomology class $-$.

Markman proves the following statements regarding $\gamma_{g,v}([\CE_v],[\CE'_v])$ when $\BK = \BZ$

\begin{enumerate}
\item[(a)] $\gamma_{g,v}:= \gamma_{g,v}([\CE_v],[\CE'_v])$ is independent of the choice of universal classes $[\CE_v]$ and $[\CE'_v]$.
\item[(b)] $\gamma_{g,v}$ is a graded $\BK$-algebra automorphism.
\end{enumerate}

Consider the subgroup 
\[
\mathrm{Spin}(S_A^+)_v \subset \mathrm{Spin}(S_A^+)
\]
preserving $v$. By \cite[Corollary 8.4]{Markman3}, every $\gamma_{g,v}$ with $g\in \mathrm{Spin}(S_A^+)_v$ is a graded ring automorphism of $H^\ast(M_{v,A}, \BZ)$ and the relation
\begin{equation}\label{composition}
\gamma_{g_1,v}\circ \gamma_{g_2,v} = \gamma_{g_1g_2,v}
\end{equation}
holds. Therefore, we have a group homomorphism (Markman's monodromy operators)
\begin{equation}\label{mon2}
    \mathrm{mon}:  \mathrm{Spin}(S_A^+)_v \to \mathrm{Aut}(H^*(M_{v,A}, \BZ)),\quad \mathrm{mon}(g) = \gamma_{g,v}.
\end{equation}

The following theorem is  proved by Markman for $\BZ$ coefficients in \cite{Markman2, Markman3}.
The extension of (\ref{mon2}) to  $\BC$-coefficients is crucial in the proof of Theorem \ref{taut_abelian}.

\begin{thm}\label{monodromy_main} 
For $g \in \mathrm{Spin}(S_A^+\otimes\BC)_v$,
the class $\gamma_{g,v}:= \gamma_{g,v}([\CE_v],[\CE'_v])$ is independent of the choice of universal classes $[\CE_v]$ and $[\CE'_v]$.
The morphism (\ref{mon2}) extends to a group homomorphism
\[
 \mathrm{mon}:  \mathrm{Spin}(S_A^+ \otimes \BC)_v \to \mathrm{Aut}(H^*(M_{v,A}, \BC)).
\]
 Moreover, for a fixed universal class $[\CE_v]$ and any $g \in \mathrm{Spin}(S_A^+ \otimes \BC)_v$, there exists
$\alpha \in H^2(M_{v,A}, \BC)$ 
so that we have
\begin{equation}\label{univ_fam}
(g\otimes \gamma_{g,v}) \mathrm{ch}(\CE_v) = \mathrm{ch}({\CE}_v)\cup \pi_M^* \mathrm{exp}(\alpha)
\end{equation}
where $\pi_M: A\times M_{v,A} \to M_{v,A}$ is the second projection.
\end{thm}

\begin{proof}
We deduce this theorem  from the case of integral coefficients using the Zariski-density of integral points
\[
\mathrm{Spin}(S_A^+)_v \subset \mathrm{Spin}(S_A^+ \otimes \BC)_v,
\]
which follows from the Borel Density Theorem \cite{Borel} applied to the simple group
\[
\mathrm{Spin}(S_A^+)_v \cong \mathrm{Spin}(v^\perp), \quad v^\perp \subset S_A^+.
\]
The first claim, on the independence of universal classes, expresses a Zariski-closed condition on the element $g$, so that  it can be deduced from (a) above; see \cite[Lemma 3.11]{Markman2}.

For the remaining claims, as in the proof of \cite[Corollary 8.4]{Markman3}, we can reduce the theorem to the case $\mathrm{rank}(v) >0$. It follows that  the class $\alpha$ in the equation (\ref{univ_fam}) is uniquely determined  by
\[
\mathrm{rank}(v) \cdot \pi_M^\ast \alpha = c_1([\CE_v])- \left[(g\otimes \gamma_{g,v}) \mathrm{ch}([\CE_v])\right]_{\mathrm{deg}~2}.
\]
The operator $\gamma_{g,v}$  (\ref{gammagv})  arises as a grading-preserving element in
 \[
\mathrm{End}(H^*(M_{v,A}, \BC)).
\]

We need to show that for any $g\in \mathrm{Spin}(S_A^+ \otimes \BC)_v$, we have that:
\begin{enumerate}
    \item[(i)] $\gamma_{g,v}$ is an automorphism of graded $\BC$-algebras, and
    \item[(ii)] the equation (\ref{univ_fam}) holds.
\end{enumerate}

We first show (ii). By \cite[Corollary 8.4]{Markman3},  for every integral point $g \in \mathrm{Spin}(S_A^+)$, the homomorphism 
\[
g\otimes \gamma_{g,v}: H^*(A \times M_{v,A}, \BC) \to H^*(A \times M_{v,A}, \BC)
\]
sends  a universal class to a universal class in the sense of \cite[Definition 3.4]{Markman3}.  It follows that the equation (\ref{univ_fam}) holds for any integral $g$. Since this equation expresses  a Zariski-closed condition with respect to the parameter $g$, we obtain (\ref{univ_fam}) by the density of integral points in $\mathrm{Spin}(S_A^+ \otimes \BC)_v$. 


As to (i),  we deduce, again, from  Zariski-density, that $\gamma_{g,v}$ is a graded $\BC$-algebra  homomorphism for any $g\in \mathrm{Spin}(S_A^+ \otimes \BC)_v$. The endomorphism $\gamma_{g,v}$ being an automorphism is not a Zariski closed condition. However, to prove $\gamma_{g,v} \in \mathrm{Aut}(H^*(M_{v,A},\BC))$, it suffices to show that 
\begin{equation}\label{eq24}
\gamma_{g,v} \circ \gamma_{g^{-1}, v} = \mathrm{id}, \quad g\in \mathrm{Spin}(S_A^+ \otimes \BC)_v.
\end{equation}
For integral $g$, the equation (\ref{eq24}) follows from (\ref{composition}),
\[
\gamma_{g,v} \circ \gamma_{g^{-1},v} = \gamma_{g\cdot g^{-1} , v} = \gamma_{\mathrm{id},v} = \mathrm{id}.
\]
Since (\ref{eq24}) is a Zariski closed condition, we conclude that it holds for any $g$, and (i) follows.
\end{proof}


\subsection{Fourier--Mukai transforms}\label{Sec2.5} 
Sections 2.5 to 2.7 are devoted to proving Theorem \ref{taut_abelian}. 
We first treat the special case where, if we let  $E$ and $E'$ be two elliptic curves, then we set
\begin{equation}\label{special}
A= E \times E', \quad \beta = \sigma + nf,
\end{equation}
where  $\sigma = [\mathrm{pt} \times E']$ and $f = [E \times \mathrm{pt}]$ are the classes of a section and a fiber with respect to the elliptic fibration $p: A \to E'$. {Note that  $\beta^2=2n$.}

By \cite[Theorem 3.15]{Yo}, we have an isomorphism of  moduli spaces
\begin{equation}\label{FM1}
\CM_{\beta,A} \cong M_{v_n, A}(= A^{[n]} \times \hat{A}),\quad v_n = (1,0,-n)
\end{equation}
induced by a relative Fourier--Mukai transform with respect to the elliptic fibration $p: A \to E'$. Our strategy is to compare a universal rank 1 torsion free sheaf on $A\times M_{v_n, A}$ with a universal 1-dimensional sheaf on $A\times \CM_{\beta,A}$ under the isomorphism (\ref{FM1}). Then we reduce  the special case (\ref{special})
of Theorem \ref{taut_abelian} to Theorem \ref{taut_Hilb}.


The proof of \cite[Theorem 3.15]{Yo} shows that the support of a sheaf $\CF \in \CM_{\beta, A}$ is the union of a section and several fibers (with multiplicities). Hence the Hilbert--Chow morphism (\ref{Hilb-Ch}) is exactly the morphism
\begin{equation}\label{HC2}
p_n\times q: A^{[n]} \times \hat{A} \to E'^{(n)} \times E.
\end{equation}
Here we use the identification between $\hat{A}$ and $A$ for the abelian surface $A= E \times E'$, and $q$ is the projection to the first factor.  Under the identification
\[
A \times A = (E \times E) \times (E'\times E'),
\]
the kernel for the relative Fourier--Mukai transform associated with the elliptic fibration $p: A \to E'$ is given by
\begin{equation}\label{rel_P}
\CP'_E = \CP_E \boxtimes \CO_{\Delta_{E'}}
\end{equation}
with
\[
\CP_E = \CO_{E\times E}(\Delta_E - o_{E}\times E - E\times o_E)
\]
the normalized Poincar\'e line bundle. Recall the universal family $\CE_n$ on $A \times M_{v_n,A}$ defined in (\ref{univ_ideal}). A universal family $\CF_\beta$ of 1-dimensional Gieseker-stable sheaves on $A \times \CM_{\beta,A}$ is induced by $\CE_n$ and $\CP_E'$ under the isomorphism (\ref{FM1}) as follows. 

For $\chi$ given in (\ref{given_d}), we define
\[
N= \CO_A\left(\chi \cdot [E\times o_{E'}]\right) = p^*\CO_{E'}(\chi \cdot o_{E'}) \in \mathrm{Pic}(A).
\]
By \cite[Theorem 3.15]{Yo} together with its proof, every closed point in $\CM_{\beta,A}$ is given by 
\[
\left(\phi_{\CP'_E}(I_Z\otimes L) \otimes N\right) [1] \in \Coh(A)
\]
where $\phi_{\CP'_E}$ is the Fourier--Mukai transform with  kernel (\ref{rel_P}), and $I_Z\otimes L \in M_{v_n,A}=A^{[n]} \times \hat{A}$  is a uniquely determined  rank one torsion free sheaf on $A$. In particular, the sheaf
\begin{equation}\label{univ2}
\CF_\beta = {R\pi_{13}}_\ast \left( \pi_{12}^\ast (\pi_1^*N\otimes{\CP'_{E}}) \otimes^{\BL} \pi_{23}^\ast \CE_n \right)[1] \in \Coh(A \times \CM_{\beta,A})
\end{equation}
is a universal family on $A \times M_{\beta,A}$, where $\pi_i, \pi_{ij}$ are projections from $A \times A \times M_{v_n,A}$ to the corresponding factors, and we identify $A \times M_{v_n,A}$ with $A \times \CM_{\beta,A}$ via (\ref{FM1}).

For notational convenience, we define the shifted normalized universal family
\begin{equation}\label{Normalized_F}
\CF^{n}_\beta = {R\pi_{13}}_\ast \left( \pi_{12}^\ast {\CP'_{E}} \otimes^{\BL} \pi_{23}^\ast \CE_n \right) = \CF_\beta \otimes \pi_A^\ast N^{-1} [-1] 
\end{equation}
on $A\times \CM_{\beta,A}$, whose restriction over every closed point $\CF \in \CM_{\beta,A}$ recovers the shifted 1-dimensional coherent sheaf $\left(\CF \otimes N^{-1}\right)[-1]$ on A.

We use the universal family (\ref{univ2}) {and the class $\alpha$ (\ref{alpha1})} to serve as the ones in Theorem \ref{taut_abelian} for the special case (\ref{special}). Next, we construct a decomposition of the perverse filtration
 \begin{equation}\label{splitting1}
H^\ast(\CM_{\beta, A}, \BQ) = \bigoplus_{i,d} \widetilde{G}_iH^d(\CM_{\beta, A}, \BQ)
\end{equation}
which serves as the splitting in Theorem \ref{taut_abelian}.

A canonical splitting of the perverse filtration associated with
\[
q: \hat{A} =A \to E
\]
is given by 
\[
G'_kH^d(\hat{A}, \BQ) = \bigoplus_{i+k =d} H^{d-k}(E, \BQ)\otimes H^k(E',\BQ).
\]
We define (\ref{splitting1}) as the decomposition
\begin{equation}\label{splitting2}
\widetilde{G}_kH^\ast(\CM_{\beta, A}, \BQ) = \bigoplus_{i+j=k}  \widetilde{G}_i H^\ast(A^{[n]}, \BQ) \otimes G'_jH^\ast(\hat{A},\BQ),
\end{equation}
where
$\widetilde{G}_\ast H^\ast(A^{[n]}, \BQ)$ was defined by (\ref{GS_decomp}). Since $\widetilde{G}_\ast H^\ast(A^{[n]}, \BQ)$ splits the perverse filtration {(\ref{cansplitting15}) associated} with 
\[
p_n:  A^{[n]} \to E'^{(n)},
\]
the filtration (\ref{splitting2}) splits the perverse filtration associated with the Hilbert--Chow morphism (\ref{HC2}).

The following theorem verifies Theorem \ref{taut_abelian} for the special case (\ref{special}), where we can choose the twisting class to be  
\begin{equation}\label{alpha1}
\alpha = - \pi_A^* c_1(N) \in H^2(A\times \CM_{\beta,A}, \BQ).
\end{equation}

\begin{thm}\label{Step1}
The decomposition (\ref{splitting2}) is multiplicative, and, with 
$\alpha$ as in  (\ref{alpha1}), we have 
\begin{equation}\label{000}
\int_\gamma \mathrm{ch}^\alpha_k(\CF_\beta) \in \widetilde{G}_{k-1} H^\ast(\CM_{\beta,A}, \BQ),\quad \; 
\forall k \geq 0, \forall\gamma \in H^\ast(A, \BQ).
\end{equation}
\end{thm}

\begin{proof}
The multiplicativity of (\ref{splitting2}) follows directly from Theorem \ref{taut_Hilb} (b) and the K\"unneth
formula. We need to prove (\ref{000}), which,  { by the very definition of $\alpha$ and $\CF^n_\beta$ and by the fact that 
a cohomological shift only changes the signs of Chern classes},
 is  equivalent to 
\begin{equation}\label{1}
\int_\gamma \mathrm{ch}_k(\CF^n_\beta) \in \widetilde{G}_{k-1} H^\ast(\CM_{\beta,A}, \BQ),\quad 
\forall k\geq 0, \; \forall \gamma \in H^\ast(A, \BQ).
\end{equation}
We consider the  product  $A \times \CM_{\beta,A}$ and the multiplicative splitting 
of its cohomology
\begin{equation}\label{eqn27}
\widetilde{G}_kH^\ast(A \times \CM_{\beta,A} ,\BQ) := H^\ast(A, \BQ) \otimes \widetilde{G}_kH^\ast(\CM_{\beta,A}, \BQ).
\end{equation} This decomposition splits the perverse filtration associated with the morphism
\[
\mathrm{id}\times \pi_\beta: A \times \CM_{\beta, A} \to A \times B.
\]
The statement (\ref{1}) is equivalent to the following:
\begin{equation}\label{eqn28}
\mathrm{ch}(\CF^n_\beta) \in \bigoplus_{k\geq 1} \widetilde{G}_{k-1} H^{2k}(A\times \CM_{\beta, A}, \BQ).
\end{equation}

For a nonsingular projective variety $X$ with a decomposition $G_*H^*(X, \BQ)$ on its cohomology, we say that a class $\alpha \in H^*(X, \BQ)$ is \emph{balanced} with respect to this decomposition if 
\[
\alpha \in \bigoplus_{k} G_{k} H^{2k}(X, \BQ).
\]
We prove (\ref{eqn28}) by the following 3 steps.

\noindent\emph{Step 1.} We consider the new decomposition 
\begin{equation}\label{31}
\begin{split}
\overline{G}_k H^*(A \times \CM_{\beta, A}, \BQ)& := \bigoplus_{i+j=k} H^i(E, \BQ) \otimes H^*(E', \BQ) \otimes \widetilde{G}_jH^*(\CM_{\beta,A}, \BQ)\\
& = \bigoplus_{i+j=k} \widetilde{G}_i H^\ast(A \times A^{[n]}, \BQ) \otimes G'_jH^\ast(\hat{A}, \BQ).
\end{split}
\end{equation}
of the cohomology $H^\ast(A\times \CM_{\beta,A}, \BQ)$. It is a multiplicative decomposition splitting the perverse filtration associated with the morphism
\[
p \times \pi_\beta: A\times \CM_{\beta,A} \to E' \times B.
\]
We first show that the class
\[
\mathrm{ch}(\CE_n) \in H^\ast(A \times M_{A, v_n}, \BQ) = H^*(A \times \CM_{\beta,A}, \BQ)
\]
is balanced with respect to the splitting (\ref{31}). Here the last identification is given by (\ref{FM1}).

Recall the expression (\ref{univ_ideal}) of the universal family $\CE_n$. By Theorem \ref{taut_Hilb} (a), the class
\[
\mathrm{ch}(\CI_n) \in  H^{\ast}(A\times A^{[n]}, \BQ)
\]
is balanced with respect to the splitting (\ref{splitting0}). Hence via pulling back along the projection
\[
\pi_{12}: A\times \CM_{\beta,A} = A\times A^{[n]} \times \hat{A} \to A \times A^{[n]},
\]
we obtain that the class 
\[
\begin{split}
\pi_{12}^\ast \mathrm{ch}(\CI_n) =  \mathrm{ch}(\CI_n) \boxtimes 1_{\hat{A}} & \in H^\ast(A \times A^{[n]}, \BQ) \otimes H^0(\hat{A}, \BQ)\\  & \subset H^*(A\times \CM_{\beta,A}, \BQ)
\end{split}
\]
is also balanced with respect to (\ref{31}). A direct calculation shows that the class $\pi_{13}^\ast c_1(\CP)$ is balanced. Hence we conclude from the multiplicativity of (\ref{31}) that the class
\[
\mathrm{ch}(\CE_n) = \pi_{12}^\ast \mathrm{ch}(\CI_n) \otimes \pi_{13}^\ast \mathrm{ch}(\CP)=  \pi_{12}^\ast \mathrm{ch}(\CI_n) \otimes \pi_{13}^\ast \mathrm{exp}(c_1(\CP))
\]
is balanced.

\noindent \emph{Step 2.} We further consider the multiplicative decomposition  
\begin{equation}\label{32}
\widetilde{G}_k H^*(A \times A \times \CM_{\beta, A}, \BQ) := H^*(A, \BQ) \otimes \overline{G}_k H^*(A \times \CM_{\beta, A}, \BQ)
\end{equation}
of $H^*(A \times A \times \CM_{\beta, A})$, which splits the perverse filtration associated with the morphism
\[
\mathrm{id} \times p \times \pi_\beta: A \times A \times \CM_{\beta, A} \to A \times E'\times B.
\]
We show that 
\begin{equation}\label{34}
\mathrm{ch}(\pi_{12}^\ast {\CP'_{E}} \otimes^{\BL} \pi_{23}^\ast \CE_n) \in \bigoplus_{k\geq 1}\widetilde{G}_{k-1} H^{2k}(A \times A\times \CM_{\beta, A}, \BQ).
\end{equation}
Here $\pi_{ij}$ are the projections from $A\times A \times \CM_{\beta,A}$ to the corresponding factors; compare the notation of (\ref{Normalized_F}).

In fact, by Step 1, we obtain that the class
\[
\pi_{23}^* \mathrm{ch}(\CE_n) = 1_A \boxtimes \mathrm{ch}(\CE_n) \in H^0(A, \BQ) \otimes H^*(A\times \CM_{\beta,A}, \BQ)
\]
is balanced with respect to (\ref{32}). The relative Poincar\'e sheaf (\ref{rel_P}) can be expressed as
\begin{equation}\label{233}
\pi_{13}^\ast \mathrm{ch}(\CP'_E) = \pi_{13}^\ast \mathrm{ch}(\CP_E)\cup \pi_{13}^* \mathrm{ch}(\CO_{\Delta_{E'}}).
\end{equation}
The class $\pi_{12}^\ast \mathrm{ch}(\CP_E)$ is balanced via a direct calculation, and 
\[
\pi_{12}^*\mathrm{ch}(\CO_{\Delta_E'}) = \pi_{12}^* c_1(\CO_{\Delta_{E'}}) \in \widetilde{G}_0 H^2(A\times A \times \CM_{\beta,A}, \BQ).
\]
Hence by (\ref{233}) and the multiplicativity of (\ref{32}), we get
\[
\pi_{12}^* \mathrm{ch}(\CP'_E)\in \bigoplus_{k\geq 1} \widetilde{G}_{k-1}H^{2k}(A\times A \times \CM_{\beta,A}).
\]
This yields (\ref{34}) since
\[
\mathrm{ch}(\pi_{12}^\ast {\CP'_{E}} \otimes^{\BL} \pi_{23}^\ast \CE_n) = \pi_{12}^*\mathrm{ch}(\CP'_E) \cup \pi_{23}^\ast \mathrm{ch}(\CE_n).
\]

\noindent \emph{Step 3.} We finish the proof of (\ref{eqn28}). 

Recall the expression (\ref{Normalized_F}) of $\CF^n_\beta$. Appying the Grothendieck--Riemann-Roch formula to the projection
\[
\pi_{13}: A \times A \times \CM_{\beta, A} \to A \times \CM_{\beta,A},
\]
we obtain that
\[
\mathrm{ch}(\CF^n_\beta) = {\pi_{13}}_\ast \mathrm{ch}(\pi_{12}^\ast {\CP'_{E}} \otimes^{\BL} \pi_{23}^\ast \CE_n).
\]
Equivalently, the class $\mathrm{ch}(\CF^n_\beta)$ corresponds to the K\"unneth factor of the class
\[
\mathrm{ch}(\pi_{12}^\ast {\CP'_{E}} \otimes^{\BL} \pi_{23}^\ast \CE_n) 
\]
in the subspace
\[
H^4(A, \BQ) \otimes H^*(A \times \CM_{\beta,A}, \BQ) \subset H^*(A\times A \times \CM_{\beta, A}, \BQ).
\]
Here $H^4(A, \BQ)$ is spanned by the point class in the second factor of the product $A \times A \times \CM_{\beta, A}$. Hence (\ref{eqn28}) follows from (\ref{34}).
\end{proof}

\subsection{Perverse filtrations and $H^2(\CM_{\beta,A}, \BQ)$}
In this section, we assume that $A$ is any abelian surface and $\beta$ is an ample curve class satisfying {$\beta^2= v_n^2=2n$}. Recall from Section \ref{Sec2.1} that $\CM_{\beta,A}$ is the moduli space of Gieseker-stable sheaves with respect to the primitive Mukai vector
\begin{equation}\label{vbeta}
v_\beta = (0, \beta, \chi) \in S_A^+
\end{equation}
with dimension
\[
\mathrm{dim}(\CM_{\beta,A}) = v_\beta^2+2 = 2n+2.
\]
Our goal is to prove Proposition \ref{prop2.4}, which is a variant of Proposition \ref{prop1.1} for $\CM_{\beta,A}$.  More precisely, we introduce a \emph{canonical} class
 \begin{equation}\label{ample_base}
 L_\beta \in H^2(\CM_{\beta,A} ,\BQ)
 \end{equation}
to characterize the perverse filtration $P_{\star} H^*(\CM_{\beta, A}, \BQ)$ associated with the morphism
\[
\pi_\beta: \CM_{\beta, A} \to B. 
\]

Let $\CF_0 \in \CM_{\beta, A}$ be a general point on the moduli space. We consider the  morphisms
defined by considering determinants of coherent sheaves
\[
\mathrm{det}: \CM_{\beta,A} \to \hat{A}, \quad \CF \mapsto \mathrm{det}(\CF)\otimes \mathrm{det}(\CF_0)^\vee 
\]
and
\[
\hat{\mathrm{det}}: \CM_{\beta,A} \to A, \quad \CF \mapsto \mathrm{det}(\phi_\CP(\CF)) \otimes \mathrm{det}(\phi_\CP(\CF))^\vee.
\]
Here $\phi_\CP : D^b\Coh(A) \to D^b\Coh(\hat{A})$ is the Fourier--Mukai transform induced by the normalized Poincar\'e line bundle $\CP$ on $A \times \hat{A}$. By \cite[(4.1) and Theorem 4.1]{Yo}, the Albanese morphism for $\CM_{\beta,A}$ with respect to the reference point $\CF_0$ is
\[
\mathrm{alb} = \hat{\mathrm{det}} \times \mathrm{det}: \CM_{\beta,A} \to A \times \hat{A}.
\]
The closed fiber over the origin $0 \in A\times \hat{A}$, denoted by
\[
K_{\beta,A} = \mathrm{alb}^{-1}(0) \subset \CM_{\beta, A},
\]
is a holomorphic symplectic vairiety of Kummer type; see \cite[Theorem 0.2]{Yo}. It parametrizes 1-dimensional sheaves $\CF \in \CM_{\beta,A}$ satisfying
\[
 \mathrm{det}(\CF)= \mathrm{det}(\CF_0) \in \hat{A},\quad   \mathrm{det}(\phi_\CP(\CF)) = \mathrm{det}(\phi_\CP(\CF)) \in \hat{\hat{A}} = A.
\]
The restriction of the Hilbert--Chow morphism (\ref{Hilb-Ch}) to the subvariety $K_{\beta, A}$ is a Lagrangian fibration
\[
\pi'_\beta: K_{\beta, A} \to  |\CO_A(\beta)| =  \BP^{n-1}.
\]
Here $ |\CO_A(\beta)|$ denotes the linear system associated with the divisor  
\[
\mathrm{supp}(\CF_0) \subset A.
\]

We consider the following commutative diagram,
\begin{equation}\label{diagram1}
    \begin{tikzcd}
K_{\beta, A} \times A \times \hat{A}\arrow{r}{h} \arrow{d}{\pi'_\beta \times \mathrm{pr}_A} & \CM_{\beta,A}\arrow{d}{\pi_\beta} \\
 \BP^{n-1}\times A \arrow{r}{h'} & B.
\end{tikzcd} 
\end{equation}
Here the horizontal arrows are defined by
\[
h(\CF, a, L) = {t_a}_* \CF \otimes  L,\quad h'(C, a)= {t_a}_\ast C
\]
with $t_a: A \xrightarrow{\simeq} A$ the translation by $a$. Since both $h$ and $h'$ are finite and surjective, the map
\begin{equation*}
    h^*: H^k(\CM_{\beta, A}, \BQ) \to H^k(K_{\beta, A} \times A \times \hat{A}, \BQ)
\end{equation*}
is injective by the projection formula. Moreover, it is an isomorphism when $k=2$ since
\[
\mathrm{dim}\,H^2(\CM_{\beta,A}, \BQ) = \mathrm{dim}\,H^2(A^{[n]}\times \hat{A}, \BQ)= \mathrm{dim}\,H^2(K_{\beta, A} \times A \times \hat{A}, \BQ)
\]
by G\"ottsche's formula \cite{Go1} for the Betti numbers of the Hilbert schemes and the generalized Kummer varieties.

By \cite[Theorem 0.1 and (1.6)]{Yo}, the correspondence induced by $\mathrm{ch}(\CF_\beta) \in H^*(A \times \CM_{\beta,A})$ together with the restriction map
\[
H^2(\CM_{\beta,A}, \BQ) \to H^2(K_{\beta,A}, \BQ)
\]
gives an isometry between $v_\beta^\perp \subset S_A^+\otimes \BQ$ and $H^2(K_{\beta,A}, \BQ)$  (endowed with the Beauville-Bogomolov-Fujiki form)
\[
\theta: v_\beta^\perp  \xrightarrow{\simeq}  H^2(K_{\beta,A}, \BQ).
\]
{Therefore, the pullback $h^*$ induces an isometry 
 (see also \cite[Theorem 4.1.(3) eq. (4.3)]{Yo})}
\begin{equation}\label{eqn39}
\varphi: H^2(\CM_{\beta, A}, \BQ) \xrightarrow{\simeq} v_\beta^\perp \oplus H^2(A\times \hat{A}, \BQ),
\end{equation}
which, by \cite[Section 8]{Markman2}, is  ${\rm Spin} (S_A^+\otimes \BQ)_{v_\beta}$-equivariant.
{
\begin{rmk}\label{compatibilityspin} The equivariance means  that 
Markman's monodromy operators $\gamma_{g,v_\beta}$ (\ref{mon2}) acting on 
$H^2(\CM_{\beta, A}, \BQ)$ are identified, via $\varphi$, with the action of $g$ on the
right-hand side of (\ref{eqn39})  described in Section \ref{spin8reps}. We may say that $\varphi$ is compatible with $g$ and $\gamma_{g,v_\beta}$.
\end{rmk}
}

We consider the following classes 
\begin{equation}\label{Lbeta}
\begin{split}
    l_\beta & = \theta\left((0,0,1)\in v_\beta^\perp\right) \in H^2(K_{\beta,A} ,\BQ), \\
    L_\beta & = \varphi^{-1} \left( (0,0,1) +  (\beta\boxtimes 1_{\hat{A}}) \right) \in H^2(\CM_{\beta, A}, \BQ).
\end{split}
\end{equation}

\begin{prop}\label{prop2.4}
We have  
\[
P_kH^m(\CM_{\beta,A}, \BQ) = \left(\sum_{i\geq 1} \mathrm{Ker}\left(L_\beta^{(n+1)+k+i-m}\right) \cap \mathrm{Im}(L_\beta^{i-1}) \right) \cap H^m(\CM_{\beta,A}, \BQ).
\]
\end{prop}

\begin{proof}
We denote by 
\[
\pi_K = \pi'_\beta \times \mathrm{pr}_A : K_{\beta, A} \times A \times \hat{A} \to \BP^{n-1}\times A
\]
the left vertical map of the diagram (\ref{diagram1}). 

By \cite[Example 3.1]{Markman4}, we have  
\[
l_\beta = {\pi'_\beta}^\ast \CO_{\BP^{n-1}}(1) \in H^2(K_{\beta,A}, \BQ).
\]
Hence 
\[
h^*L_\beta = l_\beta \boxtimes 1_{A\times\hat{A}}+ 1_{K_{\beta,A}}\boxtimes (\beta\otimes 1_{\hat{A}}) \in H^2(K_{\beta, A} \times A \times \hat{A},\BQ)
\]
is the pullback of an ample class on the base $\BP^{n-1}\times A$ via the morphism $\pi_K$. Since $h$ and $h'$ are finite, Proposition \ref{prop2.4} is deduced immediately from the diagram (\ref{diagram1}) and the following lemma.
\end{proof}

\begin{lem}
We consider a commutative diagram 
\begin{equation*}
    \begin{tikzcd}
X\arrow{r}{h} \arrow{d}{f} & X'\arrow{d}{f'} \\
 Y \arrow{r}{h'} & Y'
\end{tikzcd} 
\end{equation*}
where all the varieties are nonsingular and irreducible, the horizontal morphisms are finite and surjective, and the vertical morphisms are proper. We assume 
\[
\mathrm{dim}(X) = \mathrm{dim}(X') = 2 \mathrm{dim}(Y) = 2\mathrm{dim}(Y') =2R
\]
for some positive integer $R$, and that there is a class $\alpha \in H^2(X', \BQ)$ satisfying
\[
h^* \alpha = f^* L
\]
with $L \in H^2(Y, \BQ)$ an ample class on $Y$. Then the perverse filtration $P_\star H^*(X', \BQ)$ associated with $f'$ can be expressed as
\begin{equation}\label{2019}
P_kH^m(X', \BQ) = \left(\sum_{i\geq 1} \mathrm{Ker}(\alpha^{R+k+i-m}) \cap \mathrm{Im}(\alpha^{i-1}) \right) \cap H^m(X', \BQ).
\end{equation}
\end{lem}

\begin{proof}
The composition of the following maps is multiplication by the degree of $h$
\[
\BQ_{X'} \to h_*h^* \BQ_{X'} = h_*\BQ_X \to \BQ_{X'},
\]
where the first map is the adjunction and the last map is the trace map. This composition is nonzero since $h$ has non-zero degree, so it yields the splitting
\begin{equation}\label{48}
h_* \BQ_X = \BQ_{X'} \oplus \CG, \quad \CG \in \mathrm{D}^b_c(X'),
\end{equation}
which further induces the embedding 
\begin{equation}\label{2020}
h^*: H^*(X', \BQ) \hookrightarrow H^*(X, \BQ) 
\end{equation}
as a  direct summand. After pushing forward (\ref{48}) along $f': X' \to Y'$, we obtain
\[
{^\mathfrak{p} \CH}^k(\mathrm{Rf'}_\ast h_\ast \BQ_X) = {^\mathfrak{p} \CH}^k(\mathrm{Rf'}_\ast \BQ_{X'}) \oplus {^\mathfrak{p} \CH}^k(\mathrm{Rf'}_\ast \CG).
\]
On the other hand, by the $t$-exactness of the finite morphism $h'$, we have
\[
{^\mathfrak{p} \CH}^k(\mathrm{Rf'}_\ast h_\ast \BQ_X) = {^\mathfrak{p} \CH}^k(h'_*\mathrm{Rf}_\ast \BQ_X)  =h'_* {^\mathfrak{p} \CH}^k(\mathrm{Rf}_\ast \BQ_X).
\]
Hence the pullback (\ref{2020}) induced by (\ref{48}) is filtered strict with respect to the perverse filtrations associated with $f$ and $f'$, and we have 
\begin{equation}\label{filer_strict}
P_kH^d(X, \BQ) \cap \mathrm{Im}(h^*) \subset P_kH^d(X', \BQ),\quad \forall k,d.
\end{equation}
We deduce (\ref{2019}) from (\ref{filer_strict}) and Proposition \ref{prop1.1} (applied to $f: X\to Y$).
\end{proof}

\subsection{Proof of Theorem \ref{taut_abelian}}\label{Section2.7} In this section, we complete the proof of Theorem \ref{taut_abelian} by combining the tools developed in the previous sections.

Let $A$ be an abelian surface and $\beta$ be an ample curve class. A \emph{perverse package} associated with the pair $(A, \beta)$ is defined to be a triple
\[
 ([\CF_\beta], \alpha_\beta, L_\beta),
\]
where $[\CF_\beta]$ is a universal class on $A \times \CM_{\beta,A}$, $L_\beta \in H^2(\CM_{\beta,A}, \BC)$ is the class introduced in Section 2.6, and $\alpha \in H^2(A\times \CM_{\beta,A}, \BC)$ is of the type (\ref{alpha}).

An isomorphism between the perverse packages associated with $(A, \beta)$ and $(A', \beta')$ is the following:
\begin{enumerate}
    \item[(i)] An  isomorphism of  graded $\BC$-vector spaces 
    \[
    g: H^*(A, \BC) \xrightarrow{\simeq} H^*(A', \BC).
    \]
    \item[(ii)] An isomorphism of graded  $\BC$-algebras
    \[
    \phi: H^*(\CM_{\beta,A}, \BC) \xrightarrow{\simeq} H^*(\CM_{\beta',A'}, \BC).
    \]
    \item[(iii)] The isomorphisms $g$ and $\phi$ satisfy that
    \[
    \begin{split}
        (g\otimes \phi) \mathrm{ch}^{\alpha_\beta} (\CF_\beta) & = \mathrm{ch}^{\alpha_{\beta'}}(\CF_{\beta'}),\\
        \phi(L_\beta) & = L_{\beta'}.
    \end{split}
    \]
\end{enumerate}
Isomorphisms of perverse packages form an equivalence relation. We call two pairs $(A, \beta)$ and $(A', \beta')$ \emph{perversely isomorphic} if there is a perverse package associated with $(A,\beta)$ isomorphic to a perverse package associated with $(A', \beta')$.

\begin{prop}\label{prop2.5}
Assume that $(A, \beta)$ is perversely isomorphic to $(A', \beta')$. If Theorem \ref{taut_abelian} holds for $(A, \beta)$, then it also holds for $(A', \beta')$.
\end{prop}

\begin{proof}
By Proposition \ref{prop2.4}, the condition 
\[
\phi(L_\beta) = L_{\beta'}
\]
implies that the perverse filtrations on $H^*(\CM_{\beta,A}, \BC)$ and $H^*(\CM_{\beta', A'}, \BC)$ are identified via $\phi$. Let 
\[
H^*(\CM_{\beta,A}, \BC) = \bigoplus_{k,d} \widetilde{G}_k H^d(\CM_{\beta,A}, \BC)
\]
be the splitting of the perverse filtration for $\CM_{\beta, A}$ satisfying
\[
\int_\gamma \mathrm{ch}^{\alpha_\beta}_k(\CF_\beta) \in \widetilde{G}_{k-1}H^*(\CM_{\beta,A}, \BC).
\] 
Then the decomposition 
\[
\
\widetilde{G}_k H^d(\CM_{\beta',A'}, \BC) = \phi(\widetilde{G}_k H^d(\CM_{\beta,A}, \BC))
\]
also splits the perverse filtration for $\CM_{\beta',A'}$, and the corresponding tautological classes lie in its correct pieces by  condition (iii) in the definition of isomorphism of perverse package.
\end{proof}

If $(A, \beta)$ is perversely isomorphic to $(A', \beta')$, then condition (ii) implies that
\begin{equation}\label{beta2beta2}
\beta^2 = \mathrm{dim}(\CM_{\beta,A}) -2 = \mathrm{dim}(\CM_{\beta',A'}) -2 = \beta'^2.
\end{equation}
The following theorem, which we  deduce from Theorem \ref{monodromy_main}, establishes the  converse
to (\ref{beta2beta2}). In other words, perverse equivalence classes associated with pairs $(A, \beta)$ are larger than the equivalence classes
given  by the deformation types of the \emph{fibrations} $\pi_\beta: \CM_{\beta,A} \to B$; see Remark \ref{remark2.8}.

\begin{thm}\label{thm2.6}
Any two pairs $(A, \beta)$ and $(A', \beta')$ satisfying
\[
\beta^2 = \beta'^2 
\]
are perversely isomorphic.
\end{thm}

\begin{proof}

As we explain at the end of the proof, it will be sufficient to find an isomorphism of  graded $\BC$-vector spaces
\begin{equation}\label{39}
g: H^*(A, \BC)  \to H^*(A', \BC)
\end{equation}
satisfying the following two properties:
\begin{enumerate}
    \item[(a)] $g$ preserves the intersection pairing;
    \item[(b)] {$g$ satisfies the following identities}
    \begin{equation}\label{40}
    \begin{split}
        g(0,0,1) & = (0,0,1),\\
        g(1,0,0) & = (1,0,0),\\
        g(0,\beta,0) & = (0, \beta',0),
    \end{split}
    \end{equation}
\end{enumerate}
and a {graded $\BC$-algebra  isomorphism}
 \begin{equation}\label{fi}
 \phi : H^*(\CM_{\beta,A}, \BC) \to H^*(\CM_{\beta', A'}, \BC)
 \end{equation}
satisfying  condition (iii) {in the definition of isomorphism of perverse packages.}

We construct $(g, \phi)$ in two steps. 

{We note that $g_1,g_2$, and $ g=g_2\circ g_1$ constructed below automatically satisfy condition (a)
by construction.}

Firstly, since the Mukai vectors (\ref{vbeta})  $v_\beta$ and $v_{\beta'}$ are primitive with the same Mukai length, 
\[
v_\beta^2 = v_{\beta'}^2 \in \BZ,
\]
we can apply \cite[Theorem 8.3]{Markman3} (which follows from Yoshioka \cite{Yo}) to find
{isometries}
\[ g_1: H^*(A, \BZ)  \to H^*(A', \BZ) ,\quad \phi_1: H^*(\CM_{\beta,A}, \BQ) \to H^*(\CM_{\beta', A'}, \BQ)\]
such that $g_1(v_\beta) = v_{\beta'}$ and
\begin{equation}\label{eq41}
(g_1 \otimes \phi_1) \mathrm{ch}(\CF_\beta) = \mathrm{ch}(\CF_{\beta'})
\end{equation}
with $[\CF_{\beta'}]$ a universal class on $A \times \CM_{\beta,A}$. Moreover, the morphism $\phi_1$ satisfies the evident variant of Remark \ref{compatibilityspin}, \emph{i.e.}, the morphism 
\[
\varphi_{\beta'}^{-1}\phi_1 \varphi_\beta: v_\beta^\perp \oplus H^2(A\times \hat{A},\BQ) 
\to v_{\beta'}^\perp \oplus H^2(A'\times \hat{A'},\BQ)
\]
is induced by the isomtery $g_1: H^*(A, \BZ) \to H^*(A', \BZ)$; see  \cite[Proposition 8.5]{Markman2}. We may say that the identifications $\varphi$ (\ref{eqn39}) for $\beta$ and $\beta'$ are compatible with $g_1$ and $\phi_1$.

If $\beta$ and $\beta'$ have different divisibility, then $g_1$ {cannot satisfy the condition (b) above}, so that  we proceed as follows.
We choose $g_2 \in \mathrm{SO}(S_{A'}^+\otimes {\BQ})_{v_{\beta'}}$ such that the restriction
$$g_2\circ g_1|_{S_{A}^{+}}: S_{A}^{+} \otimes {\BQ} \rightarrow S_{A'}^{+}\otimes {\BQ}$$
is grading-preserving and satisfies condition (b).
We can lift $g_2$ to an element of
$\mathrm{Spin}(S_{A'}^{+}\otimes \BC)_{v_{\beta'}}$, also denoted by $g_2$, and take the associated isomorphisms
\[
\begin{split}
 g_2: &  H^*(A',\BC)  \rightarrow H^*(A',\BC) \qquad \mbox{ (cf. Section \ref{spin8reps})}
\\
\phi_2=  \gamma_{g_2,v_{\beta'}}: &   H^*(\CM_{\beta',A'},\BC)  \rightarrow H^*(\CM_{\beta',A'},\BC)
\quad
\mbox{(cf. Section \ref{Monodromy})}.
\end{split}
\]
By Theorem \ref{monodromy_main}, we have
\begin{equation}\label{eq42}
(g_2 \otimes \phi_2) \mathrm{ch}(\CF_{\beta'}) =  \mathrm{ch}^{\pi_\CM^*a}(\CF_{\beta'}), \quad {\exists} a \in H^2(\CM_{\beta',A'}, \BC).
\end{equation}

We set $g = g_2 \circ g_1$ and $\phi = \phi_2\circ \phi_1$.
Then  (\ref{eq41}) and (\ref{eq42}) imply that
\[
(g\otimes \phi)\mathrm{ch}(\CF_{\beta}) = \mathrm{ch}^{\pi_\CM^*a}(\CF_{\beta'}).
\]
In particular, for any class $\alpha_\beta \in H^2(A\times \CM_{\beta,A}, \BC)$ of type (\ref{alpha}), we have
\[
(g\otimes \phi)\mathrm{ch}^{\alpha_\beta}(\CF_{\beta}) = \mathrm{ch}^{\alpha_{\beta'}}(\CF_{\beta'}), \quad {\rm with}\; \alpha_{\beta'} = \pi_\CM^*a+ (g\otimes \phi)\alpha_\beta,
\]
where the class $\alpha_{\beta'}$ is also of type (\ref{alpha}). {The first condition appearing in (iii) is thus met.}

{It remains to verify that $\phi (L_\beta)=L_{\beta'}$.
In view of  the compatibility of the $\varphi$'s for $\beta$ and $\beta'$ with $g_1, \phi_1$, and of $\varphi$ for $\beta'$ with $g_2, \phi_2$ (cf. Remark \ref{compatibilityspin}), and in view  of the construction of $g_1$ and $g_2$, 
the isomorphism $\phi$  sends $(0,0,1) \in v_\beta^\perp$ to $(0,0,1) \in v_{\beta'}^\perp$, and sends $\beta\boxtimes 1_{\hat{A}} \in H^2(A\times \hat{A}, \BC)$ to $\beta' \boxtimes 1_{\hat{A'}} \in H^2(A'\times \hat{A'}, \BC)$. In particular, we have
$
\phi(L_\beta) = L_{\beta'}
$, and the desired condition (3) is met in its entirety.
This completes the proof of Theorem \ref{thm2.6}. }
\end{proof}

Theorem \ref{taut_abelian} follows from Theorem \ref{Step1}, Proposition \ref{prop2.5}, and Theorem \ref{thm2.6}
{(with $n:=\beta^2/2$, $v_n$ and $\beta$)}. In Section \ref{Section3.5}, we prove a strengthened version of Theorem \ref{taut_abelian} to the effect that  the decomposition $\tilde{G}_*H^*(\CM_{A,\beta}, \BC)$ is the one  induced by a variant of a construction of Deligne's. 

\begin{rmk}\label{remark2.8}
It was shown in \cite{dCM} that perverse filtrations behave well under deformations. However, for the pairs $(A, \beta)$ and $(A',\beta')$ such that $\beta$ and $\beta'$ have different divisibilities in $H^2(A, \BZ)$ and $H^2(A',\BZ)$, the fibrations $\pi_\beta$ and $\pi_{\beta'}$ are not deformation equivalent.
Therefore,  in order to reduce the proof of Theorem \ref{taut_abelian} for any pair $(A,\beta)$ to that for a special pair
\[
A= E \times E', \quad \beta = \sigma +n f,
\] 
it is essential to consider an equivalence relation weaker than deformation equivalences of fibrations $\pi_\beta$. This is our motivation for introducing  equivalences of perverse packages. 
\end{rmk}

\section{Splittings of perverse filtrations}

\subsection{Overview}
In this section, we study splittings of the perverse filtration associated with a proper surjective morphism $\pi: X\to Y$ with $X$ and $Y$ nonsingular. As an application, we strengthen Theorem \ref{taut_abelian} by requiring that the decomposition  {given by said theorem}
\begin{equation}\label{decomp6}
H^\ast(\CM_{\beta, A}, \BC) = \bigoplus_{i,d} \widetilde{G}_iH^d(\CM_{\beta, A}, \BC).
\end{equation}
is induced by a ``Lefschetz class" via the mechanisms introduced in \cite{D, dC, dCM4}; see Theorem \ref{strengthened2.1}. As discussed in Remark \ref{remark1}, this is crucial in the study of the specialization morphism (\ref{sp}) in Section 4. 

Throughout, we work with $\BQ$-coefficients except for Section \ref{Section3.5}, where we need to switch to $\BC$-coefficients in order to apply results in Section \ref{Section2.7}. However, we note that all discussions in 
Sections \ref{Lefschetz classes}--\ref{Section3.4} remain valid if we replace $\BQ$ by $\BC$.

\subsection{Lefschetz classes}\label{Lefschetz classes}
We consider the perverse filtration associated with $\pi: X \to Y$,
\begin{equation}\label{e50}
P_0H^\ast(X, \BQ) \subset P_1H^\ast(X, \BQ) \subset \dots \subset P_{2{R}}H^\ast(X, \BQ) =  H^\ast(X, \BQ),
\end{equation}
where ${R}$ is the defect of semismallness  of $\pi$ (\ref{defect}). The action of a class $\eta \in H^2(X, \BQ)$ on the cohomology $H^*(X, \BQ)$ via cup product satisfies
\begin{equation}\label{e51}
\eta: P_kH^i(X, \BQ) \to P_{k+2}H^{i+2}(X, \BQ),
\end{equation}
which further induces an action on 
\[
\BH = \bigoplus_{\star,\bullet\geq 0} \mathrm{Gr}_\star^P H^{\bullet}(X, \BQ)= \bigoplus_{\star,\bullet \geq 0} \left( P_\star H^{\bullet}(X, \BQ)/P_{\star-1}H^{\bullet}(X, \BQ) \right).
\]

 We say that  $\eta \in H^2(X, \BQ)$  is a \emph{$\pi$-Lefschetz class}  if its induced action on $\BH$ satisfies the hard Lefschetz-type condition in the sense of \cite[Assumption 2.3.1]{dC}, \emph{i.e.}, the actions
\begin{equation}\label{pi-lefschetz}
\eta^k: \mathrm{Gr}^P_{{R} -k}H^*(X, \BQ) \xrightarrow{\simeq}  \mathrm{Gr}^P_{{R}+k}H^*(X, \BQ), \quad \forall k\geq 0,
\end{equation}
are isomorphisms. As a typical example, the relative Hard Lefschetz Theorem \cite{BBD} with respect to $\pi: X \to Y$ implies that a relatively ample class for $\pi$  is a $\pi$-Lefschetz class. The following lemma gives examples of Lefschetz classes other than relatively ample classes.  

\begin{lem}\label{Lemma3.1}
Let $A$ be an abelian variety of dimension $m$. Any class $\eta \in H^2(A, \BQ)$ satisfying
\begin{equation}\label{eta1}
\eta^m \neq 0 
\end{equation}
is a Lefschetz class with respect to $\pi: A \to \mathrm{pt}$.
\end{lem}

\begin{proof} 
Let $W = H^1(A, \BQ)$. We identify $H^k(A, \BQ)$ with $\wedge^k W$, and therefore $\eta \in \wedge^2W$. The condition (\ref{eta1}) implies that $\eta$ defines a (constant) symplectic form on $W^*$. In particular, we can write $\eta = \sum_{i=1}^m e_i \wedge f_i$ under a basis $e_1,\dots, e_m, f_1, \dots, f_m$ of $W$. Hence the pair $(W\otimes \BC, \eta)$ is the tensor product of $m$ copies of the symplectic plane $(\BC^2, e \wedge f)$, for which the statement is clear.
\end{proof}

Recall the Hitchin fibration $h: \CM_{\mathrm{Dol}} \to \Lambda$. The following proposition is a numerical criterion for Lefschetz classes with respect to $h$.

\begin{prop}\label{prop3.2_Hitchin} Let $F$ be a closed fiber of $h$. A class $\eta \in H^2(\CM_{\mathrm{Dol}}, \BQ)$ is a Lefschetz class with respect to $h$ if and only if {the  following cap product with the fundamental class of the connected fiber $F$ (counting components with multiplicities) is non-trivial}:
\begin{equation}\label{eqn50}
0 \neq \eta^{\mathrm{dim(F)}} \cap [F] \in H_0(F,\BQ) \simeq \BQ.
\footnote{
We note that the condition (\ref{eqn50}) does not depend on the choice of the  fiber $F$. }
\end{equation}
\end{prop}

\begin{proof}
 We denote by $\hat{\CM}_{\mathrm{Dol}}$ the corresponding moduli space of stable $\mathrm{PGL}_r$-Higgs bundles with
\[
\hat{h}: \hat{\CM}_{\mathrm{Dol}} \to \hat{\Lambda}
\]
the Hitchin fibration. By \cite{HT1, Markman}, we have
\[
H^2( \hat{\CM}_{\mathrm{Dol}}, \BQ ) = \BQ \cdot \alpha
\]
where $\alpha$ is the first Chern class of an $\hat{h}$-ample line bundle on   $\hat{\CM}_{\mathrm{Dol}}$, and therefore is
a $\hat{h}$-Lefschetz class. By the discussion in \cite[Section 2.4]{dCHM1}, we have an isomorpism
\begin{equation}\label{112233}
H^*(\CM_{\mathrm{Dol}}, \BQ) = H^*( \hat{\CM}_{\mathrm{Dol}}, \BQ ) \otimes H^*(\mathrm{Pic}^0(C), \BQ)
\end{equation}
satisfying that
\begin{equation}\label{eqn71}
P_k H^*(\CM_{\mathrm{Dol}}, \BQ) = \bigoplus_{i+j=k} P_iH^*( \hat{\CM}_{\mathrm{Dol}}, \BQ ) \otimes H^j(\mathrm{Pic}^0(C), \BQ).
\end{equation}
Under the identification 
\[
H^2(\CM_{\mathrm{Dol}}, \BQ) = H^2( \hat{\CM}_{\mathrm{Dol}}, \BQ ) \oplus H^2(\mathrm{Pic}^0(C), \BQ) = \BQ \cdot \alpha  \oplus H^2(\mathrm{Pic}^0(C)),
\]
induced by (\ref{112233}), we can express any class $\eta \in H^2(\CM_{\mathrm{Dol}}, \BQ)$ as
\begin{equation}\label{123eta}
\eta = \mu {\alpha  \otimes 1+ 1\otimes \xi}, \quad \mu \in \BQ,~ \xi \in H^2(\mathrm{Pic}^0(C), \BQ).  
\end{equation}
Lemma \ref{Lemma3.1} {combined with \cite[Appendix]{dC}} implies that the class $\eta$ is Lefschetz if and only if 
\begin{equation}\label{eqn61}
\mu \neq 0, \quad \xi^g\neq 0.
\end{equation}
Therefore, it suffices to show that (\ref{eqn61}) is equivalent to the condition (\ref{eqn50}).

We  first consider traceless Higgs bundles
\[
\CM^0_{\mathrm{Dol}} = \{(\CE, \theta) \in \CM_{\mathrm{Dol}}: \mathrm{trace}(\theta)=0\} \subset \CM_{\mathrm{Dol}}.
\]
Recall from \cite[Proposition 2.4.3]{dCHM1} that there is a finite morphism
\[
q: \check{\CM}_{\mathrm{Dol}} \times \mathrm{Pic}^0(C)   \to \CM^0_{\mathrm{Dol}}
\]
with $\check{\CM}_{\mathrm{Dol}}$ the corresponding Dolbeault  moduli  space of stable $\mathrm{SL}_r$-Higgs bundles. The preimage of a closed fiber $F \subset \CM^0_{\mathrm{Dol}}$ of the restricted Hitchin fibration $h|_{\CM^0_{\mathrm{Dol}}}$ is the product
\[
q^{-1}(F) =  \check{F} \times \mathrm{Pic}^0(C),
\]
where $\check{F}$ {is the corresponding  closed fiber} of the $\mathrm{SL}_n$ Hitchin fibration. The pullback of a class (\ref{123eta}) along
\[
\check{F} \times \mathrm{Pic}^0(C) \hookrightarrow \check{\CM}_{\mathrm{Dol}} \times \mathrm{Pic}^0(C)   \xrightarrow{q} \CM^0 \hookrightarrow \CM_{\mathrm{Dol}}
\]
is of the type 
\[
\check{\eta} = \mu \check{\alpha}\otimes 1 + 1\otimes \xi \in H^2(\check{F}, \BQ) \oplus H^2(\mathrm{Pic}^0(C), \BQ) \subset H^2(\check{F} \times \mathrm{Pic}^0(C), \BQ)
\]
with $\check{\alpha}$ an ample class on $\check{F}$. Since $q: \check{F} \times \mathrm{Pic}^0(C) \to F$ is finite
{and surjective}, the condition (\ref{eqn50}) is equivalent to
\[
0 \neq \check{\eta}^{\mathrm{dim}(\check{F} \times \mathrm{Pic}^0(C))} \cap [\check{F}] = \check{\eta}^{\mathrm{dim}(F)}\cap [\check{F}],
\]
which, in turn,  is  equivalent to (\ref{eqn61}). This completes the proof.
\end{proof}

\subsection{The first Deligne splitting}\label{First Deligne Splitting}

Given a projective morphism $\pi : X \to Y$ and a $\pi$-ample line bundle  on $X$, Deligne \cite{D} 
constructs three splittings of the direct image $R\pi_* \BQ_X$ (resp. $R\pi_* \mathrm{IC}_X$ if $X$ is singular), which induce splittings of the perverse filtration on the (resp. intersection) cohomology groups $H^*(X,\BQ)$. In this paper, we need the variant \cite{dC} of this construction
where one starts with a $\pi$-Lefschetz class $\eta \in H^2(X,\BQ)$, \emph{i.e.} one that does not necessarily satisfy the relative Hard Lefschetz Theorem in the derived category $D^b_c(Y)$, but nonetheless satisfies the cohomological version (\ref{pi-lefschetz}).
We need only the first of the resulting three splittings, which we name the \emph{first Deligne splitting}
(\cite{D,dCM4,dC}).


According to \cite{dC}, the first Deligne splitting of the perverse filtration (\ref{e50})   
\[
H^*(X,\BQ) = \bigoplus_{i}G_iH^*(X, \BQ)
\]
 associated with the $\pi$-Lefschetz class $\eta$ (\emph{i.e.} (\ref{pi-lefschetz}) holds)
can be described using \emph{only} the action of $\eta$ on $H^*(X,\BQ)$.  We explain this description explicitly as follows.

For $i  \geq 0$, we let $\mathrm{Gr}^P_i:=P_iH^*(X,\BQ)/P_{i-1}H^*(X,\BQ)$ denote the graded spaces, which are subquotients of cohomology, and let $G_i:= G_i H^*(X,\BQ)$ denote the corresponding image of $\mathrm{Gr}^P_i$ via the first Deligne splitting, i.e. these are the summands in the last paragraph, which are subspaces of cohomology that we want to characterize.

For $0 \leq k \leq {R}$, let 
\[
\mathrm{Gr}^P_k\supseteq Q'_k:= 
\ker \{\eta^{{R}-k+1}: \mathrm{Gr}^P_{k}\to \mathrm{Gr}^P_{2{R} -k +2}\}
\]
be the associated graded-primitive spaces (here $\eta$ acts on the graded spaces of the perverse filtration on cohomology).
Let $Q_k \subseteq G_k$ be the image of $Q'_k$ via the first Deligne splitting. We have
$G_k= \sum_{i \geq 0} \eta^i Q_{k-2i}$ for $0 \leq k \leq {R}$, and $G_{{R}+\kappa}=\eta^k G_{{R}-\kappa}$
(here $\eta$ acts on cohomology).

It follows that, in order to have a complete description of the first Deligne splitting that involves solely
the action of $\eta$ in cohomology, we only need to describe $Q_k \subseteq H^*(X,\BQ)$ in such terms.

We describe $Q_k$ following \cite[Section 2.7]{dCM4}. Note that its context is the one of the fist Deligne splitting arising from working in the derived category, but the proof works verbatim in the present context of cohomology acted upon by  a $\pi$-Lefschetz class. By (\ref{e51}), we have
\[
\eta^j P_k H^*(X, \BQ) \subset P_{k+2j}H^*(X, \BQ).
\]
Let $\Phi_0 (\eta)$ be the composition of the morphisms
\[
\Phi_0 (\eta) : P_kH^*(X, \BQ) \xrightarrow{\eta^{{ R}-k+1}} P_{2{ R}-k+2}H^*(X, \BQ) \to \mathrm{Gr}_{2{ R}-k+2}^PH^*(X, \BQ),
\]
where the first map is cup product and the second map is the projection to the graded piece. We obtain the sub-vector space
\[
\mathrm{Ker}(\Phi_0 (\eta)) \subset P_kH^*(X, \BQ).
\]
The morphisms $\Phi_m (\eta)$ are defined inductively for $m\geq 0$ as follows:
\[
\Phi_m (\eta) : \mathrm{Ker}(\Phi_{m-1} (\eta)) \xrightarrow{\eta^{{ R} -k+m}} P_{2{R} -k+2m}H^*(X, \BQ) \to \mathrm{Gr}_{2 { R} -k+2m}^PH^*(X, \BQ).
\]
Therefore for fixed $k$, we obtain a sequence of sub-vector spaces
\[
\cdots \subset \mathrm{Ker}(\Phi_1(\eta)) \subset \mathrm{Ker}(\Phi_0 (\eta) ) \subset P_kH^*(X, \BQ). 
\]
According to \cite[Proposition 2.7.1]{dCM4}, we have the desired description
\[
Q_k = Q_k (\eta)= \mathrm{Ker}(\Phi_k (\eta)) \subset P_kH^*(X, \BQ), \quad \forall 0 \leq k \leq { R}.
\]

The description of the first Deligne splitting yields immediately the following comparison lemma, which expresses a kind of functoriality for the first Deligne splitting.

{
\begin{lem}\label{comparison1}
Let $X_i$ $(i=1,2)$ be nonsingular varieties with proper surjective morphisms $f_i: X_i \to Y_i$, and let $P_\star H^*(X_i, \BQ)$ be the corresponding perverse filtrations.  Let  
\[
\phi: H^*(X_1, \BQ) \to H^*(X_2, \BQ)
\]
be a morphism
of graded $\BQ$-algebras satisfying 
\begin{enumerate}
    \item[(a)] $\phi(P_kH^d(X_1, \BQ)) \subset P_kH^d(X_2, \BQ)$;
    \item[(b)] $\phi(\eta_1) = \eta_2$,  with {$\eta_i \in H^2(X_i, \BQ)$ an $f_i$-Lefschetz class for $i= 1,2$.}
\end{enumerate}
Then we have
\[
\phi(G_kH^*(X_1, \BQ)) \subset G_k H^*(X_2, \BQ), \quad \forall k \geq 0
\]
with $ H^*(X_i,\BQ)= \oplus_{k\geq 0}G_kH^*(X_i, \BQ)$ the first Deligne splitting associated with $\eta_i$, $i=1,2$.
\end{lem}

\begin{proof}
It follows from the description of the first Deligne splittings 
\[H^*(X_i, \BQ) = \bigoplus_{0 \leq k \leq r} 
 \bigoplus_{j \geq 0} 
 \left(\eta^j_i Q_{k-2j}(\eta_i)  \bigoplus \eta^{r-k+j}_i Q_{k-2j}  (\eta_i)
 \right) \quad i=1,2
\] 
summarized above, and the fact that the sub-vector spaces
\[
\mathrm{Ker}(\Phi_k (\eta_i)) = Q_k (\eta_i) \subset P_kH^*(X_i, \BQ)
\]
are preserved under the $P$-filtered morphism $\phi$.
\end{proof}
}

\begin{rmk}
Since cohomologically, the Hitchin fibration 
\[
h: \CM_{\mathrm{Dol}} \to  \Lambda,
\]
behaves like the product of the fibrations $\hat{h}: \hat{\CM}_{\mathrm{Dol}} \to \hat{\Lambda}$ and $\mathrm{Pic}^0(C) \to \mathrm{pt}$ via the ring isomorphism
\[
H^*(\CM_{\mathrm{Dol}}, \BQ) =H^*(\hat{\CM}_{\mathrm{Dol}}, \BQ) \otimes H^*(\mathrm{Pic}^0(C), \BQ)
\]
given in \cite[Section 2.4]{dCHM1}, we see from the proof of Proposition \ref{prop3.2_Hitchin}
 {(cf. (\ref{eqn61}))} together with \cite[Appendix]{dC} that the first Deligne splitting associated with any 
$h$-Lefschetz class, i.e. a class of the form $\mu \alpha\otimes 1+ 1\otimes \xi$ with $\mu\neq 0$ and $\xi^g \neq 0$, has the form
\begin{equation}\label{Hitchin_splitting}
G_k H^*(\CM_{\mathrm{Dol}}, \BQ) = \bigoplus_{i+j=k} G_iH^*(\hat{\CM}_{\mathrm{Dol}}, \BQ) \otimes H^j(\mathrm{Pic}^0(C), \BQ), 
\end{equation}
where $G_iH^*(\hat{\CM}_{\mathrm{Dol}}, \BQ)$ is the first Deligne splitting associated with the $\hat{h}$-Lefschetz class 
\[
\alpha \in H^2(\hat{\CM}_{\mathrm{Dol}}, \BQ).
\]
\begin{rmk}\label{dolglcansplit}
In particular, the splitting (\ref{Hitchin_splitting}) does not depend on the choice of an  $h$-Lefschetz class.
\end{rmk}
For  every genus $g \geq 2$, we expect, and actually prove in the $g=2$ case, that (\ref{Hitchin_splitting}) serves as the splitting in Conjecture \ref{P=W2}.
\end{rmk}

\subsection{Semismall maps and Hilbert schemes}\label{Section3.4}
In this section, we study the situation where our morphism $\pi: X \to Y$ can be factored as
\[
X \xrightarrow{f} Z \xrightarrow{g} Y
\]
with $f: X \to Z$ semismall and surjective; in particular, $f$ is generically finite and $\dim (X)=\dim (Z)$. We further assume that there are closed irreducible sub-varieties $Z_i \subset Z$ such that the decomposition theorem for $f$ takes  the form of a canonical  finite direct sum decomposition
\begin{equation}\label{decomp9}
Rf_* \BQ_X = \bigoplus_i \mathrm{IC}_{Z_i}[-\mathrm{dim}(X)],
\end{equation}
where each $\mathrm{IC}_{Z_i}$ is the (perverse) intersection cohomology complex of $Z_i$. Note that we have the following identity concerning intersection cohomology groups
$\mathrm{IH}^d(Z_i, \BQ)=H^{d-\dim Z_i}(Z_i, \mathrm{IC}_{Z_i})$.
One of the $Z_i$ is the total variety $Z$. The restriction of $g$ to each $Z_i$ yields the morphism
\[
g_i: Z_i \to Y_i \subset Y.
\]
We deduce from (\ref{decomp9}) a canonical decomposition of the cohomology of $X$
\begin{equation}\label{decomp8}
H^d(X, \BQ) = \bigoplus_i \mathrm{IH}^{d-c_i}(Z_i, \BQ), {\quad c_i = \dim X - \dim Z_i={\rm codim}_X Z_i.}
\end{equation}

Let ${R}$ be the defect of semismallness of $\pi$ and, for each $i$, let ${R_i}$ be the defect of semismallness of $g_i:Z_i \to Y_i$.

\begin{prop}\label{proposition3.5}
Let  $\alpha \in H^2(Z, \BQ)$ satisfy that, for every $i$,   the restriction
\[
\alpha_i = \alpha|_{Z_i} \in H^2(Z_i, \BQ)
\]
is a $g_i$-Lefschetz class  with associated  first Deligne splitting 
\[
\mathrm{IH}^*(Z_i, \BQ) = \bigoplus_k G_k \mathrm{IH}^*(Z_i, \BQ).
\]
Then $\eta = f^*\alpha \in H^2(X, \BQ)$ is a $\pi$-Lefschetz class, whose associated first Deligne splitting is,
under the identification (\ref{decomp8}), given by
\begin{equation}\label{decomp7}
{G_kH^d(X, \BQ) = \bigoplus_i G_{k-R+{R_i}}\mathrm{IH}^{d-c_i}(Z_i, \BQ).}
\end{equation}
\end{prop}

\begin{proof}{Recall that we have defined the perverse filtration $P$ on $H^*(X,\BQ)$ concentrated in the interval $[0,{2R}]$,
and similarly,  for every $i$, the perverse filtration  $P$ on $\rm{IH}^*(Z_i,\BQ)$ is concentrated in the interval
$[0,{2R_i}]$. It follows that, according to  (\ref{decomp9}) and (\ref{decomp8}),  the direct summand $\rm{IH}^*(Z_i,\BQ)$ contributes to  
$H^*(X,\BQ)$ in perversities in the interval ${[R-R_i,R+R_i]}$.

We apply the decomposition theorem to the composition
\[
X \to Z \to Y
\]
and obtain that the perverse filtration $P_\star H^*(X, \BQ)$ can be expressed in terms of the perverse filtrations $P_\star \mathrm{IH}^*(Z_i, \BQ)$ under (\ref{decomp8}), \emph{i.e.},
\[
P_kH^d(X, \BQ) = \bigoplus_i P_{k-{R} +{R_i}}\mathrm{IH}^{d-c_i}(Z_i, \BQ).
\]

By \cite[Remark 4.4.3]{dCM0}, the action of $\eta=f^*\alpha$ on the l.h.s. of (\ref{decomp8}), is 
the direct sum of the actions of the classes $\alpha_i=\alpha_{|Z_i}$ on the summands of the r.h.s..

Since every $\alpha_i$ is a $g_i$-Lefschetz class, 
we deduce that $\eta$ is a $\pi$-Lefschetz class with associated first Deligne splitting decomposition (\ref{decomp7}).}
\end{proof}

{Next, we show that, given a fibered abelian surface 
\[p:A=E\times E' \to E'\]
  as in Section \ref{Section1.5}, 
the splitting (\ref{canonical_split}) of the perverse filtration on $H^*(A^{[n]}, \BQ)$ associated with the natural morphism  $p_n: A^{[n]} \to E'^{(n)}$, 
 is the first Deligne splitting induced by a  $p_n$-Lefschetz class. 

The morphism $p_n: A^{[n]} \to E'^{(n)}$ admits the natural factorization 
\begin{equation}\label{factorization_Hilb}
A^{[n]} \xrightarrow{f} A^{(n)} \xrightarrow{g} E'^{(n)}
\end{equation}
where the Hilbert--Chow morphism $f$ is semismall \cite{Go2}. There are canonical morphisms
\[
\kappa_\nu: A^{(\nu)} \to A^{(n)}
\]
together with a canonical stratification of $A^{(n)}$ indexed by the partitions $\nu$ of $n$,
\[
A^{(n)} = \bigcup_\nu A_\nu, \quad \overline{A_\nu} = \mathrm{Im}(\kappa_\nu) \subset A^{(n)}.
\]
Note that  the resulting morphism $\kappa_\nu:A^{(\nu)} \to \overline{A_\nu}$ is the normalization of the target,
so that 
\begin{equation}\label{hisih}
H^*(A^{(\nu)},\BQ)= {\rm IH}^*(\overline{A_\nu},\BQ).
\end{equation}

By \cite[Theorem 3]{Go2}, the decomposition theorem associated with $f$ takes the form
analogous to (\ref{decomp9}),
\[
Rf_* \BQ_{A^{[n]}} = \bigoplus_{\nu} \mathrm{IC}_{\overline{A_\nu}} [-2n],
\]
and the restriction of $g: A^{(n)} \to E'^{(n)}$ to each $\overline{A_\nu}$ is described by the commutative diagrams
\begin{equation}\label{diagram9}
\begin{tikzcd}
A^{(\nu)} \arrow{r}{\kappa_\nu}  \arrow[swap]{d} &  \overline{A_\nu}\arrow[r, hookrightarrow]  \arrow[swap]{d}{p_\nu} &  A^{(n)}  \arrow{d} \\
{E'}^{(\nu)} \arrow{r} & \overline{{E'_\nu}} \arrow[r, hookrightarrow] & E'^{(n)}.
\end{tikzcd}
\end{equation}

We consider the $p$-ample class
\[
\alpha = \mathrm{pt} \otimes 1_{E'} \in H^2(E, \BQ)\otimes H^0(E', \BQ) \subset H^2(A, \BQ)
\]
of $p:A \to E'$ which induces a Lefschetz class 
\begin{equation}\label{alpha(n)}
\alpha^{(n)} \in H^2(A^{(n)}, \BQ)
\end{equation}
with respect to $A^{(n)} \to E'^{(n)}$. It further induces a Lefschetz class
\[
\alpha^{(\nu)} \in H^2(A^{(\nu)}, \BQ)
\]
with respect to $A^{(\nu)} \to E'^{(\nu)}$ for every partition $\nu$ of $n$. By the diagram (\ref{diagram9}), the pullback of $\alpha^{(n)}$ to every $A^{(\nu)}$ via $\kappa_\nu$ coincides with $\alpha^{(\nu)}$.

By keeping in mind that  ${\rm codim}_{A^{(n)}} \overline{A_\nu}=2n -2l(\nu)$, that the defects of semismallness of the morphisms  $p_n$ and $p_\nu$ are $n$ and $l(\nu)$, respectively, and the identity (\ref{hisih}),
we obtain the following corollary by applying Proposition \ref{proposition3.5} to the factorization (\ref{factorization_Hilb}).
}

\begin{cor}\label{cor3.6}
The decomposition (\ref{canonical_split}) is the first Deligne splitting induced by the 
$(p_n: A^{[n]} \to E'^{(n)})$-Lefschetz class  ($f$ as in (\ref{factorization_Hilb}), $\alpha^{(n)}$
as in (\ref{alpha(n)}))
\[
\eta_{A^{[n]}} = f^*\alpha^{(n)} \in H^2(A^{[n]}, \BQ).
\]
\end{cor}

\subsection{A strengthened version of Theorem \ref{taut_abelian}}\label{Section3.5}
We study the decomposition (\ref{decomp6}) constructed for Theorem \ref{taut_abelian}. In the special case (\ref{special}), we have
\[
\CM_{\beta, A} \cong A^{[n]} \times \hat{A}, \qquad A=E \times E',
\]
with the morphism $\pi_\beta: \CM_{A,\beta} \to B$ given by the morphism $p_n\times q$ in (\ref{HC2}). Corollary \ref{cor3.6} 
implies that the decomposition (\ref{splitting2}) is the first Deligne splitting associated with the 
$(p_n\times q)$-Lefschetz class 
\[
\eta_{A,\beta} = \eta_{A^{[n]}} \boxtimes 1_{\hat{A}} + 1_{A^{[n]}}\boxtimes \left(1_E \boxtimes \mathrm{pt} \right) \in H^2(A^{[n]} \times \hat{A}, \BQ) =  H^2(\CM_{\beta, A}, \BQ).
\]

By Theorem \ref{thm2.6}, a pair $(A', \beta')$ with $\beta'^2=2n$ is perversely isomorphic to the special pair $(A, \beta)$ given by (\ref{special}). So there is a graded isomorphism
\[
\phi: H^*(\CM_{\beta, A}, \BC) \xrightarrow{\simeq} H^*(\CM_{\beta', A'}, \BC) 
\]
of $\BC$-algebras preserving the corresponding perverse filtrations, and we obtain a decomposition (\ref{decomp6}) as
\begin{equation}\label{996}
\widetilde{G}_kH^d(\CM_{\beta',A'}, \BC) = \phi(\widetilde{G}_kH^d(\CM_{\beta',A'}, \BC));
\end{equation}
see the proof of Proposition \ref{prop2.5}. Hence
\[
\eta_{\beta',A'} = \phi(\eta_{\beta,A}) \in H^2(\CM_{\beta',A'}, \BC)
\]
is a $\pi_{\beta'}$-Lefschetz class, and Lemma \ref{comparison1} implies that the decomposition (\ref{996}) is the first Deligne splitting associated with $\eta_{\beta',A'}$. 

This gives the following strengthened version of Theorem \ref{taut_abelian}, which, we stress, is about any pair $(A,  \beta)$ with $v_\beta = (0, \beta, \chi)$ primitive, not just the special cases (\ref{special}).

\begin{thm}\label{strengthened2.1}
Theorem \ref{taut_abelian} holds for a splitting $\widetilde{G}_\ast H^\ast(\CM_{\beta,A}, \BC)$ given by the first Deligne splitting associated with a {suitable $\pi_\beta$-Lefschetz class $\eta_\beta$}, where $\pi_\beta: \CM_{\beta,A} \to B$  is the morphism (\ref{Hilb-Ch}).
\end{thm}

\section{Topology of Hitchin fibrations}\label{se4}

\subsection{Overview} Throughout the section, we assume $C$ is a nonsingular projective integral curve of genus $g$ embedded into an abelian surface $A$. We study {a kind of} specialization morphism
\begin{equation}\label{sp}
    \mathrm{sp}^!: H^*(\CM_{\beta, A}, \BQ) \to H^*(\CM_{\mathrm{Dol}}, \BQ)
\end{equation}
where $\CM_{\mathrm{Dol}}$ is the moduli of stable Higgs bundles of rank $r$ and Euler characteristic $\chi$, and 
\[
\beta = r[C] \in H_2(A, \BZ).
\]
Then we deduce Theorems \ref{genus_2}, \ref{P=W_even}, and \ref{P=W_odd} from Theorem \ref{taut_abelian} via the properties of the  morphism (\ref{sp}) which we establish hereafter.

\subsection{Deformation to the normal cone}\label{Section 4.2}

The moduli space $\CM_{\mathrm{Dol}}$ of stable Higgs bundles with rank $r$ and Euler characteristic $\chi$ can be realized as the moduli space of $1$-dimensional Gieseker-stable sheaves $\CF$ on the cotangent bundle  surface $T^*C$ with (cf. see \cite{BNR})
\[
 [\mathrm{supp}(\CF)] = r[C] \in H_2(T^*C, \BZ), \quad \chi(\CF) = \chi.
\]
In the following, we describe $\CM_{\mathrm{Dol}}$ as the ``limit'' of a trivial family of $\CM_{\beta,A}$. A similar construction using $K3$ surfaces was considered in \cite{DLL}.

Assume $T = \BP^1$ and $T^\circ = \BP^1 \setminus 0 = \BC$. Let 
\begin{equation}\label{normal_cone}
 \CS = \mathrm{Bl}_{C\times 0}(A \times T) \setminus (A\times 0) \to T
\end{equation}
be the total space of the deformation to the normal cone associated with the embedding $j_C:C\hookrightarrow A$. The central fiber of (\ref{normal_cone}) is
\[
\CS_0 = T^*C \to 0 \in T,
\]
and the restriction over $T^\circ$ is a trivial fibration
\[
A \times T^\circ \to  T^\circ \subset T.
\]

We associate to $\beta$ a family of homology classes 
\begin{equation}\label{beta_t}
\beta_t = r[C] \in H_2(\CS_t, \BZ).
\end{equation}
Let $\CM \rightarrow T$ be the (coarse) relative moduli space which parametrizes, for each $t \in T$, pure one-dimensional Gieseker-stable sheaves $\CF$ on $\CS_t$ with $\chi(\CF) =\chi$ and such that the support of $\CF$ is a proper subscheme in the class $\beta_t$. Similarly, let $\CB \rightarrow T$ denote the component of the relative Hilbert scheme which parametrizes Cartier divisors in $\CS_t$ with proper support in the class $\beta_t$.  Following \cite[Section 5.3]{GIT}, we have a proper morphism 
\begin{equation*}
\begin{tikzcd}
\CM \ar [rr,"h_T"]\ar[rd]& & \CB\ar[ld]\\
 &T&
\end{tikzcd}
\end{equation*}  
over the base $T$, which, on the level of points, sends a sheaf $\CF$ on $\CS_t$ to its Fitting support, viewed as an element in $\CB_t$.
Clearly both $\CM$ and $\CB$ become trivial families when restricted over $T^\circ$,
\begin{equation*}
\begin{tikzcd}
\CM_{\beta,A} \times T^\circ \ar [rr,"h_{T^\circ}"]\ar[rd]& & B \times T^\circ \ar[ld]\\
 &T^\circ&.
\end{tikzcd}
\end{equation*}  
The fibers over $0 \in T$ recover the moduli space $\CM_{\mathrm{Dol}}$, the Hitchin base $\Lambda$, and the Hitchin fibration (\ref{Hitchin_fib}).

The following lemma seems standard; we include a brief proof  since we are not aware of an adequate reference.
\begin{lem}\label{lem4.1}
\begin{enumerate}

\item[(i)] Both $\CM$ and $\CB$ are irreducible varieties and smooth over $T$.
\item[(ii)] There exists a universal class 
\[
[\CF_T] \in K_{\mathrm{top}}(\CM\times_{T} \CS)
\]
whose restriction to each fiber gives a universal class for $\CM_t \times \CS_t$.
\end{enumerate}

\end{lem}
\begin{proof}

For part (i), we first argue for $\CB$. For notational convenience, we define
\[
\tilde{g} = \mathrm{dim}(\CB_t) = r^2(g-1)+1, \quad \forall~t \in T.
\]
We observe that the closure $\overline\CB^{\circ}$ in $\CB$ of 
\[
\CB^{\circ} = B \times T^\circ
\]
is irreducible of dimension $\tilde{g}+1$, so that the intersection 
\[
\overline{\CB^{\circ}}\cap \CB_0 \subset \CB_0
\]
must have at least dimension $\tilde{g}$ at every closed point.  This intersection is non-empty since it contains the divisor $r[C]$ which clearly deforms to the generic fiber. Since $\CB_0$ is irreducible of dimension $\tilde{g}$, we have $\CB_0 \subset \overline{\CB^{\circ}}$. Therefore $\CB$ is irreducible. This implies the flatness of $\CB$ over the nonsingular curve $T$. Furthermore, since its fibers are all nonsingular, the flatness further implies that the morphism $\CB \to T$ is smooth.

The same argument applies for $\CM$. The only thing to check is that there exists a point in the central fiber $\CM_0$ which deforms to the generic fiber.  By the smoothness of $\CB/T$,  we can choose a nonsingular curve $Z_0 \subset \CS_0$ which deforms to $Z_t \subset \CS_t$.  Any line bundle on $Z_0$ with Euler characteristic $\chi$ can be deformed as well and this gives the desired point in $\CM_0$.

For part (ii), we follow the same argument from \cite[Section 3.1]{Markman5}. We denote by $\mathfrak{M}$ the relative moduli stack of stable sheaves on $\CS$, which is a $\mathbb{G}_m$-gerbe over $\CM$. It suffices to show that this gerbe is \emph{topologically} trivializable, so that we can pull back the topological $K$-theory class of the universal family on $\mathfrak{M}\times_T\CS$ to $\CM\times_T \CS$ via a section.  To construct a trivialization, we take the nonsingular 3-fold $\CS'= \mathrm{Bl}_{C\times 0}(A \times T)$ which contains $\CS$ as an open subset,
\[
\CS \subset \CS'.
\] 
Let $\CM_{\CS'}$ be the (coarse) moduli space of stable sheaves on $\CS'$ with Chern character given by $w =\mathrm{ch}({i_t}_\ast \CF) \in H^*(\CS', \BQ)$, where 
\[
i_t: \CS_t \xrightarrow{\simeq} A\times t \subset \CS', \quad t\neq 0
\]
is the inclusion of a nonsingular fiber, and $\CF \in \CM$ is a coherent sheaf supported on $\CS_t$. The moduli space $\CM_{\CS'}$ contains $\CM$ as an open subvariety, therefore it is enough to show that the corresponding $\mathbb{G}_m$-gerbe $\mathfrak{M}_{\CS'} \rightarrow \CM_{\CS'}$ is topologically trivial. As in \cite[Section 3.1]{Markman5}, a trivialization can be constructed with a topological K-theory class in $\CS'$ whose pairing with $w$ is $1$. For this, we can use the pullback to $\CS'$ of a class on $A$, whose pairing with $(0 ,r[C], \chi)$ is $1$, via the natural projection
\[
\CS' = \mathrm{Bl}_{C\times 0}(A\times T) \to A. \qedhere
\]
\end{proof}

\subsection{Specializations}\label{section4.3} Specialization morphisms with respect to perverse filtrations have been studied systematically in \cite{dCM, sp}. We provide a concrete description of the specialization morphism in our setting as follows.

As before, we assume $T = \BP^1$ and $T^\circ = \BP^1 \setminus 0$. Let $f: W \to T$ be a smooth morphism, whose restriction over $T^\circ$ is a trivial product
\[
f^\circ: W^\circ = W_t \times T^\circ \rightarrow T^\circ, \quad \forall~ t\neq 0,
\]
with $W_t$ proper.
By Deligne's Global Invariant Cycle Theorem, see  \cite[Theorem 1.7.1]{dCM1} for example, the restriction 
\[
\mathrm{res}_t:  H^*(W, \BQ) \to  H^*(W^\circ,\BQ)=H^*(W_t, \BQ)
\]
is surjective for $t \neq 0$. Let 
\[
\alpha_1, \alpha_2 \in H^*(W, \BQ)
\]
be two liftings of a class $\alpha \in H^*(W_t, \BQ)$ with $t\neq 0$. Since $H^*(W^\circ,\BQ)=H^*(W_t,\BQ)$, the long exact sequences for the pairs $(W,W_t)$ and $(W, W^\circ)$
are isomorphic. Since $W$ is nonsingular, the relative cohomology for $(W,W^\circ)$ can be identified with the Borel-Moore homology of $W_0$. It follows that
\begin{equation}\label{en52}
\alpha_1 - \alpha_2 = {i_0}_* \gamma, 
\end{equation}
for some class $\gamma$
 in the Borel--Moore homology of $W_0$.  Here $i_0: W_0 \hookrightarrow W$ is the closed embedding of the central fiber, and we identify the Borel--Moore homology and the cohomology of the nonsingular $W$ in the equation (\ref{en52}). The excess intersection formula \cite[Corollary 6.3]{Fulton}, {together with the triviality of the normal bundle $N_{W_0/W}$,} implies that we have
 \[
\mathrm{res}_0(\alpha_1) - \mathrm{res}_0(\alpha_2) = i_0^*{i_0}_* \gamma = c_1(N_{W_0/W})\cap \gamma =0.
\]
We define the \emph{specialization morphism}\footnote{In the literature, often, one defines a specialization morphism  $\mathrm{sp}^*$ in the opposite direction, as it is dictated by the morphism $\sigma: i_0^* \to \psi_f$ with $\psi_f$ nearby cycle-functor. In the present set-up, since $W_0$ is not proper, such an arrow $sp^*$ does not exist; 
by dualizing $\sigma$, one sees that an arrow in the opposite direction always exists, but if $W_0$ is singular, for example, this arrow is not a familiar object.}
\begin{equation}\label{sp!}
\mathrm{sp}^!: H^*(W_t, \BQ) \rightarrow H^*(W_0, \BQ), \quad t\neq 0
\end{equation}
as
\begin{equation}\label{sp1}
\mathrm{sp^!}(\alpha) = \mathrm{res}_0(\tilde{\alpha}) \in H^*(W_0, \BQ)
\end{equation}
where $\tilde{\alpha} \in H^*(W, \BQ)$ is any lifting of $\alpha$. The discussion above implies that the class (\ref{sp1}) does not depend on the lifting $\tilde{\alpha}$. Since $\mathrm{res}_0$ is a homomorphism of $\BQ$-algebras, so is $\mathrm{sp}^!$  (\ref{sp!}).

\begin{example}\label{example1}
Let $\CS \to T$ be the relative surface (\ref{normal_cone}) associated with the embedding $j_C: C \hookrightarrow A$. We have the specialization morphism
\begin{equation}\label{2333}
\mathrm{sp}^!: H^*(A, \BQ) \to H^*(T^*C, \BQ) = H^*(C, \BQ).
\end{equation}
By the definition of $\mathrm{sp}^!$, the morphism (\ref{2333}) is given by the pullback along
\[
T^*C  \hookrightarrow \BP(T^*C \oplus \CO_C) \hookrightarrow \mathrm{Bl}_{C\times 0}(A \times T) \to A\times T \to A.
\]
In particular, we have
\[
\mathrm{sp}^!(\gamma)  = j_C^* \gamma \in H^*(T^*C, \BQ) = H^*(C, \BQ),\quad \forall~ \gamma \in H^*(A, \BQ). 
\]
\end{example}

Now we discuss the interaction between $\mathrm{sp}^!$ and perverse filtrations. Let $W/T$ be as above. We consider a commutative diagram

\begin{equation*}
\begin{tikzcd}
W \ar [rr,"h"]\ar[rd]& & V\ar[ld]\\
 &T&
\end{tikzcd}
\end{equation*}  
where the morphism $h: W \to V$ is proper of relative dimension $d$. For every $t \in T$, there is a perverse filtration $P_\star H^*(W_t, \BQ)$ associated with the morphism $h_t : W_t \to V_t$.

\begin{prop}\label{Prop3.1}
The specialization morphism (\ref{sp!}) preserves\footnote{{In fact, the proof shows more, namely that $\mathrm{sp}^!$ is filtered strict.}} the perverse filtrations, \emph{i.e.}, 
\[
\mathrm{sp}^!\left( P_kH^*(W_t, \BQ ) \right) \subset P_k H^*(W_0, \BQ), \quad \forall k.
\]
\end{prop}

\begin{proof}
We show that there exist splittings of the perverse filtrations
\[
H^*(W_t, \BQ) = \bigoplus_{i}G_iH^*(W_t, \BQ), \quad \forall~ t \in T
\]
such that the splitting is constant for $t \in T^o$, and
\begin{equation*}
\mathrm{sp}^!\left(G_iH^*(W_t ,\BQ)\right) \subset G_iH^*(W_0, \BQ).
\end{equation*}

We apply the decomposition theorem \cite{BBD} to the morphism $h: W \to V$, and fix an isomorphism
\[
\phi: Rh_\ast \BQ_W [\dim{W} - d] \xrightarrow{\simeq} \bigoplus_{i=0}^{2d} \CP_V^i [-i].
\]
Here, the $\CP_V^i$ are semisimple perverse sheaves on $V$. By the discussion in Section \ref{sec1.3}, the isomorphism $\phi$ induces a splitting of the perverse filtration on $H^*(W, \BQ)$ associated with $h$. By the smoothness of $f$ and the compatibility of vanishing cycles with the derived pushforward $\mathrm{Rh}_*$, we have $\varphi(\mathrm{Rh}_\ast \BQ_W) = 0$.  As a result, \cite[Corollary 3.1.6]{dCM} shows that the restriction of $\CP_V^i$ to a closed fiber over $t\in T$ is a shifted perverse sheaf,
\[
\CP_{V,t}^i [-1] = \mathrm{res}_t^*\CP_V^i [-1] \in \mathrm{Perv}(V_t), \quad \forall~ t\in T.
\]
Therefore, the restriction $\phi_t$ of the decomposition $\phi$ induces  splittings of the perverse filtration
\[
H^*(W_t, \BQ) = \bigoplus_{i \geq 0} \phi_t (H^*(V_t, \CP_{V,t}^i)), \quad \forall~ t\in T,
\]
and we conclude from the definition of $\mathrm{sp}^!$ that
\[
\mathrm{sp}^! \left(\phi_t(H^k(V_t, \CP_{V,t}^i))\right) \subset \phi_0(H^k(V_0, \CP_{V,0}^i)).
\]
This completes the proof.
\end{proof}

Now we consider the family
\[
h_T: \CM \to  \CB
\]
over $T$ constructed in Section \ref{Section 4.2}, especially Lemma \ref{lem4.1}. Proposition \ref{Prop3.1} implies that the specialization morphism
\[
\mathrm{sp}^!: H^*(\CM_{\beta, A}, \BC) \to H^*(\CM_{\mathrm{Dol}}, \BC)
\]
preserves the perverse filtrations. Let $\widetilde{G}_*H^*(\CM_{A,\beta}, \BQ)$ be a splitting given by Theorem \ref{strengthened2.1}. In particular,  it is the first Deligne splitting induced by a  { $(\pi_\beta=h_{t\neq 0})$-}Lefschetz class\footnote{Note that this class was denoted by $\eta_{\beta',A'}$ in Section \ref{Section3.5}, since $A$ was special there.}
\[
\eta_{\beta,A} \in H^2(\CM_{A,\beta}, \BC).
\]

\begin{prop}\label{prop4.4}
The class
\[
\eta_0 = \mathrm{sp}^!\left( \eta_{\beta,A}\right) \in H^2(\CM_{\mathrm{Dol}}, \BC)
\]
is a Lefschetz class for the Hitchin fibration $h: \CM_{\mathrm{Dol}} \to \Lambda$. Assume further that \[
H^*(\CM_{\mathrm{Dol}}, \BC) = \bigoplus_{k,j}\widetilde{G}_k H^j(\CM_{\mathrm{Dol}}, \BC)\] is the first Deligne splitting associated with $\eta_0$. Then we have {(cf. Theorem \ref{strengthened2.1})}
\begin{equation}\label{prop4.4eq}
\mathrm{sp}^!\left( \widetilde{G}_kH^d(\CM_{A,\beta}, \BC) \right) \subset \widetilde{G}_k H^d(\CM_{\mathrm{Dol}}, \BC).\footnote{Here we use $\BC$-coefficients since the Lefschetz class $\eta_{A,\beta}$ lies in $H^2(\CM_{\beta,A}, \BC)$; see Section \ref{Section3.5}.}
\end{equation}
\end{prop}

\begin{proof}
Let $F_{\beta,A} \subset \CM_{\beta,A}$ be a 
closed fiber of the morphism $\pi_\beta: \CM_{\beta,A} \to B$. We have
\[
 \mathrm{dim}(F_{\beta,A})= \frac{1}{2}\mathrm{dim}(\CM_{\beta,A}) = \tilde{g}.
\]
Since the fiber class 
\[
[F_{\beta,A}] \in H^{2\tilde{g}}(\CM_{\beta,A}, \BC)
\]
lies in $P_0H^{2\tilde{g}}(\CM_{\beta,A}, \BC)$ by Proposition \ref{prop1.1}, and $\eta_{\beta,A}$ is a Lefschetz class, we obtain that
\[
\eta_{\beta,A}^{\tilde{g}}\cap{[F_{\beta,A}]} = \eta_{\beta,A}^{\tilde{g}} \cup [F_{\beta,A}] \neq 0
\]
by the hard Lefschetz condition. Hence after specialization, we have
\[
{\eta_0^{\mathrm{dim(F_0)}}} \cap [F_0]  = \eta_{\beta,A}^{\tilde{g}}\cap{[F_{\beta,A}]} \neq 0
\]
with $F_0$ a closed fiber of $h: \CM_{\mathrm{Dol}} \to \Lambda$. We conclude from Proposition \ref{prop3.2_Hitchin} that $\eta_0$ is a Lefschetz class with respect to $h$. The inclusion  (\ref{prop4.4eq}) then follows from Lemma \ref{comparison1} {(applied to $\phi = \mathrm{sp}^!$)}.
\end{proof}

\begin{rmk}\label{remark1}
Proposition \ref{prop4.4} implies that $\mathrm{sp}^!(G_*H^*(\CM_{\beta,A}, \BC))$ splits the restricted perverse filtration $P_*H^*(\CM_{\mathrm{Dol}}, \BC) \cap \mathrm{Im(sp^!)}$, \emph{i.e.},
\[
P_kH^*(\CM_{\mathrm{Dol}}, \BC) \cap \mathrm{Im(sp^!)} = \bigoplus_{i\leq k} \mathrm{sp}^!\left( \widetilde{G}_iH^*(\CM_{\beta,A}, \BC)\right).
\]
In general, for an arbitrary splitting $G_*H^*(\CM_{\beta,A}, \BC)$ of the perverse filtration $P_*H^*(\CM_{\beta,A}, \BC)$, it may not be true that
\[
\mathrm{sp}^!\left({G}_iH^*(\CM_{\beta,A}, \BC) \right) \cap \mathrm{sp}^!\left({G}_jH^*(\CM_{\beta,A}, \BC) \right) = \{0\}, \quad \forall i\neq j.
\]
It is thus crucial, in our approach to the results of this paper, to realize the splitting 
\[
H^*(\CM_{\beta,A}, \BC) = \bigoplus_{k,j}\widetilde{G}_kH^j(\CM_{\beta,A}, \BC)
\]
as the first  Deligne splitting associated with a Lefschetz class as in Theorem \ref{strengthened2.1}.
\end{rmk}

\subsection{Normalized classes} Our purpose is to apply the specialization morphism $\mathrm{sp}^!$ to the tautological classes. We first discuss some properties of the normalized classes introduced in Section \ref{sec_0.3}.

Following \cite{HRV, HT1, Shende}, the cohomology 
\begin{equation}\label{eqn89}
H^*(\CM_{\mathrm{Dol}}, \BQ) = H^*(\CM_B, \BQ)
\end{equation}
can be understood by the associated $\mathrm{PGL}_n$-character variety $\hat{\CM}_{\mathrm{B}}$. More precisely, let
\[
H^*(\CM_B, \BQ) \cong H^*(\hat{\CM}_{\mathrm{B}}, \BQ) \otimes H^*((\BC^*)^{2g}, \BQ)
\]
be the isomorphism established in \cite[Section 1]{HT1} and \cite[Theorem 2.2.12]{HRV}. Every class $w \in H^i(\hat{\CM}_{\mathrm{B}}, \BQ)$ can be naturally viewed as a class
\[
w = w\otimes 1 \in H^i(\CM_B, \BQ).
\]
The weights of the tautological classes associated with the universal $\mathrm{PGL}_r$-bundle $\CT$ on $C\times \CM_B$ (induced by the universal $\mathrm{GL}_r$-bundle) were calculated in \cite{Shende}, 
\begin{equation}\label{eqn52}
\int_\gamma c_k(\CT) \in {^k\mathrm{Hdg}^{i+2k-2}(\CM_B)},\quad { \forall k \geq 0}, \forall~\gamma \in H^i(C, \BQ).
\end{equation}
The following lemma deduces (\ref{taut_weight}) from (\ref{eqn52}).

\begin{lem}\label{lem3.1}
A twisted universal family $(\CU^\alpha, \theta)$ on $C \times {\CM}_{\mathrm{Dol}}$ satisfies
\begin{equation}\label{53}
\int_\gamma \mathrm{ch}^\alpha_k(\CU) \in {^k\mathrm{Hdg}^{i+2k-2}(\CM_B)},\quad \forall~\gamma \in H^i(C, \BQ),~~ \forall~ k\geq 0,
\end{equation}
if and only if $\mathrm{ch}^\alpha(\CU)$ is normalized.
\end{lem}

\begin{proof}
We define
\[
{^k\mathrm{Hdg}^{d}(C\times \CM_B)} = \bigoplus_{i+j=d} H^i(C, \BQ) \otimes {^k\mathrm{Hdg}^{j}(\CM_B)}.
\]
Then (\ref{53}) is equivalent to
\[
\mathrm{ch}^\alpha(\CU) \in \bigoplus_k {^k\mathrm{Hdg}^{2k}(C\times \CM_B)}.
\]
By a direct calculation using Chern roots, we have (recall that $r$ is the rank)
\[
\mathrm{ch}^\alpha(\CU) = \mathrm{ch}(\CT) \cup \mathrm{exp}\left(\frac{c_1(\CU)}{r}+ \alpha\right).
\]
By using the universal relations between Chern character and total Chern class, we note that (\ref{eqn52}) is equivalent to \[
\mathrm{ch}(\CT) \in \bigoplus_k {^k\mathrm{Hdg}^{2k}(C\times \CM_B)}.
\]
By \cite{Shende}, we have
\[
H^0(\CM_B, \BQ) = {^0\mathrm{Hdg}^{0}(\CM_B)}, \quad H^2(\CM_B, \BQ) = {^2\mathrm{Hdg}^{2}(\CM_B)},
\]
which implies 
\[
{^1\mathrm{Hdg}^{2}(C\times \CM_B)} = H^1(C, \BQ)\otimes H^1(\CM_B ,\BQ).
\]
Therefore, we obtain that (\ref{53}) holds if and only if
\[
\mathrm{ch}_1^\alpha(\CU)= c_1(\CU) + r\alpha \in H^1(C, \BQ) \otimes H^1(\CM_B, \BQ).
\]
This is equivalent to the condition that the class $\mathrm{ch}(\CU^\alpha)$ is normalized.
\end{proof}

Next, in parallel to Lemma \ref{lem3.1},  we give a suffcient criterion involving the perverse filtration  for a class $\mathrm{ch}^\alpha(\CU)$ to be normalized.

\begin{lem}\label{lem4.7}
Suppose we have a twisted universal family $(\CU^\alpha, \theta)$ such that, for every $k \geq 0$ and 
every $\gamma \in H^{2i}(C, \BQ)$, the tautological class 
\begin{equation*}
\int_\gamma \mathrm{ch}^\alpha_k(\CU) \in H^{2i+2k-2}(\CM_\mathrm{Dol}, \BQ),
\end{equation*}
has perversity $k$. Then $\mathrm{ch}^\alpha(\CU)$ is normalized.
\end{lem}

\begin{proof}
We have 
\[
 P_1H^2(\CM_{\mathrm{Dol}}, \BQ) = 0, \quad  P_0H^0(\CM_{\mathrm{Dol}}, \BQ)=H^0(\CM_{\mathrm{Dol}}, \BQ).
\]
The first vanishing follows from the decomposition \eqref{112233}, and the corresponding vanishing for $\hat{\CM}_{\mathrm{Dol}}$ due to the fact that its second cohomology is generated by a relative ample class which must therefore have perversity $2$.\footnote{To see that a relative ample class has perversity 2, we apply \cite[Theorem 1.4.8]{dCHM1} together with the fact that an ample class does not vanish over a general fiber of the Hitchin morphism.} The second equality is clear in view of the isomorphism $H^0 (\CM_{\mathrm{Dol}},\BQ) \cong H^0 (\Lambda, \BQ)$ via pull-back from the base of the Hitchin fibration (\ref{Hitchin_fib}).
Since by assumption, any K\"unneth factor of $\mathrm{ch}_1^\alpha(\CU)$ in $H^*(\CM_{\mathrm{Dol}}, \BQ)$ has perversity $1$, we reach the desired conclusion
\[
\mathrm{ch}_1^\alpha(\CU) \in  H^1(C, \BQ)\otimes H^1(\CM_\mathrm{Dol} ,\BQ). \qedhere
\]
\end{proof}

\begin{rmk}\label{rmk4.8}
In fact, we see from the proofs of Lemmas \ref{lem3.1} and \ref{lem4.7} that both lemmas hold if we use $\BC$-coefficients for the cohomology groups and we allow $\alpha$ to be a $\BC$-class. In particular, a $\BC$-normalized class $\mathrm{ch}^\alpha(\CU)$ is unique and rational.
\end{rmk}

\subsection{Specializations and tautological classes}
Assume $j_C: C\hookrightarrow A$ is the closed embedding of the curve in an abelian surface. Let $\CS \to T$ and $\CM \to T$ be the family of surfaces and the relative moduli space of 1-dimensional sheaves introduced in Section \ref{Section 4.2}. In the following, we construct a family of cohomology classes on $\CM$ over $T$ whose restriction to every closed fiber $\CM_t$ is a twisted tautological class.

For our purpose, we take a relative compactification $\CS \subset \CS'$ over $T$ as in the proof of Lemma \ref{lem4.1} (ii). Then $\CM$ is the {coarse} moduli space of Gieseker-stable 1-dimensional sheaves $\CF$ on $\CS'$ satisfying that $\mathrm{supp}(\CF)$ is proper and contained in $\CS_t \subset \CS_t'$ with the numerical conditions
\[
\chi(\CF) = \chi, \quad [\mathrm{supp}(\CF)] = \beta_t.
\]
Here recall that $\beta_t$ is the class (\ref{beta_t}). By the proof of Lemma \ref{lem4.1} (ii), there is a universal $K$-theory class $[\CF_T]$ on $\CS' \times_T \CM$.

Let
\[
\pi_{\CS'}: \CS' \times_T \CM \to \CS', \quad \pi_\CM: \CS' \times_T \CM \to \CM
\]
be the projections. We have the Chern character 
\[
\mathrm{ch}({\CF_T}) \in H^*(\CS' \times_T \CM, \BC)
\]
defined by the universal class $[\CF_T]$. For any class of the type
\begin{equation}\label{alphatilde}
\tilde{\alpha} = \pi_\CS^* \tilde{\alpha}_1 + \pi_\CM^* \tilde{\alpha}_2 \in H^2(\CS' \times_T \CM, \BC)
\end{equation}
with $\tilde{\alpha}_1\in H^2(\CS', \BC),~ \tilde{\alpha}_2\in H^2(\CM, \BC)$, we consider the relative twisted class
\begin{equation}\label{456}
\mathrm{ch}^{\tilde{\alpha}}(\CF_T) = \mathrm{exp}(\tilde{\alpha})\cup \mathrm{ch}(\CF_T),
\end{equation}
whose degree $2k$ parts
\[
 \mathrm{ch}_k^{\tilde{\alpha}}(\CF_T) \in H^{2k}(\CS' \times_T \CM, \BC)
\]
induce a relative tautological class
\begin{equation}\label{427}
\int_{\tilde{\gamma}} \mathrm{ch}_k^{\tilde{\alpha}}(\CF_T): = {\pi_\CM}_* \left( (\pi_{\CS'}^*\tilde{\gamma} \cup  \mathrm{ch}_k^{\tilde{\alpha}}(\CF_T))\cap [\CS' \times_T \CM] \right) \in H_*^\mathrm{BM}(\CM,\BC) =H^*(\CM, \BC)
\end{equation}
for any $\tilde{\gamma}\in H^*(\CS', \BC)$. Here we use the Poincar\'e duality in the last identity, and we require the properness of $\pi_\CM$ for the pushforward functor ${\pi_\CM}_* : H^{\mathrm{BM}}_*(\CS' \times_T \CM, \BC) \to H^{\mathrm{BM}}_*(\CM, \BC)$.

Now we check that (\ref{427}) is the desired cohomology class on $\CM$ which gives a family of twisted tautological classes over the base $T$.

For any closed point $t \in T$, base change \cite[Proposition 1.7]{Fulton} implies that the restriction of (\ref{427}) to a closed fiber $\CM_t$ is a twisted universal class of the form 
\begin{equation}\label{restriction1}
\int_{\gamma_t} \mathrm{res}_t\left( \mathrm{ch}_k^{\tilde{\alpha}} \left(\CF_T\right)\right) \in H^*(\CM_t, \BC),\quad \gamma_t \in H^*(\CS'_t, \BC).
\end{equation}
When $t \neq 0$, we have $\CS_t= \CS'_t$, and therefore (\ref{restriction1}) recovers a twisted universal class on $\CM_{\beta,A}$,
\[
\int_{\gamma_t} \mathrm{ch}_k^\alpha (\CF_\beta) \in H^*(\CM_{\beta,A}, \BC).
\]
We now calculate the class (\ref{restriction1}) for $t =0$. By \cite[Theorem 4.6.5]{HL}, there exists a coherent sheaf $\CF_0$ supported on $T^*C \times \CM_{\mathrm{Dol}}$ which represents a universal class $[\CF_0]$ on $\CS'\times \CM_{\mathrm{Dol}}$. Moreover, the coherent sheaf $\CF_0$ is supported on the universal spectral curve $\CC \subset T^*C \times \CM_{\mathrm{Dol}}$ which is proper over $\CM_{\mathrm{Dol}}$. Hence by \cite[Property 2.1]{BFM}, we have
\begin{equation}\label{support123}
\mathrm{ch}_k(\CF_0) \cap [\CS_0'\times \CM_{\mathrm{Dol}}] = \sum_i c_i [Z_i] \in \mathrm{CH}_*(\CS'_0 \times \CM_{\mathrm{Dol}})_\BQ
\end{equation}
where $c_i \in \BQ$ and $Z_i$ are closed subvarieties in $T^*C \times \CM_{\mathrm{Dol}}$ which are proper over $\CM_{\mathrm{Dol}}$ via the composition
 \begin{equation}\label{composition1}
 Z_i \hookrightarrow T^*C \times \CM_{\mathrm{Dol}} \xrightarrow{p_\CM} \CM_{\mathrm{Dol}}. 
 \end{equation}

In particular, the expression (\ref{support123}) together with the properties for $Z_i$ implies that the class 
\[
\int_{\gamma_0} \mathrm{res}_0\left( \mathrm{ch}_k^{\tilde{\alpha}} \left(\CF_T\right)\right) \in H^*(\CM_{\mathrm{Dol}}, \BC)
\]
only depends on the restrictions of the classes {$\gamma_0$  to $H^2(T^*C, \BC)$ and
 $\mathrm{res}_0(\tilde{\alpha})$ to  $H^2(T^*C, \BC)\oplus H^2(\CM_{\mathrm{Dol}}, \BC)$}. 
 
 Note that a class of the type
\begin{equation}\label{820}
{p_\CM}_*(\omega \cup \mathrm{ch}_k(\CF_0)): = \sum_i c_i\cdot {p_{\CM}}_*(\omega \cap [Z_i]) \in H^{\mathrm{BM}}_*(\CM_{\mathrm{Dol}}, \BQ) = H^*(\CM_{\mathrm{Dol}}, \BQ)
\end{equation}
is well-defined for $\omega \in H^*(T^*C \times \CM_{\mathrm{Dol}}, \BC)$ due to the properness of (\ref{composition1}). We see from the discussion above that the restriction of the class (\ref{427}) to $\CM_0 (= \CM_{\mathrm{Dol}})$ is given by the following:
\begin{equation}\label{1114}
\mathrm{res}_0\left( \int_{\tilde{\gamma}} \mathrm{ch}_k^{\tilde{\alpha}} (\CF_T) \right) = \int_{\mathrm{sp}^!(\gamma_t)} \mathrm{ch}_k^{\alpha_0}(\CF_0) \in H^*(\CM_{\mathrm{Dol}}, \BC).
\end{equation}
Here $\mathrm{sp}^!(\gamma_t) \in H^*(\CS_0, \BC) = H^*(T^*C, \BC)$, $\alpha_0 \in H^2(T^*C, \BC)\oplus H^2(\CM_{\mathrm{Dol}}, \BC)$, and the class on the r.h.s. of (\ref{1114}) is defined by (\ref{820}).

The following proposition shows that the tautological classes (\ref{taut_class}) on the Dolbeault side $\CM_{\mathrm{Dol}}$  are obtained by specializing certain other tautological classes on the compact geometry side $\CM_{\beta,A}$.

\begin{prop}\label{prop4.8}
Let 
\[
\int_{\gamma} \mathrm{ch}_k^\alpha (\CF_\beta) \in H^*(\CM_{\beta,A}, \BC)
\]
be the classes of Theorem \ref{taut_abelian} with $\gamma \in H^*(A, \BQ)$ a rational class on $A$, then we have 
\[
\mathrm{sp}^!\left( \int_{\gamma} \mathrm{ch}_k^\alpha (\CF_\beta) \right) = c(k-1, j_C^*\gamma) \in H^*(\CM_{\mathrm{Dol}}, \BQ).
\]
\end{prop}

\begin{proof}
By the definition of $\mathrm{sp}^!$ and (\ref{1114}), we have
\[
\mathrm{sp}^!\left( \mathrm{res}_t\left( \int_{\tilde{\gamma}} \mathrm{ch}_k^{\tilde{\alpha}} (\CF_T) \right)\right) = \mathrm{res}_0\left( \int_{\tilde{\gamma}} \mathrm{ch}_k^{\tilde{\alpha}} (\CF_T) \right)= \int_{\mathrm{sp}^!(\gamma_t)} \mathrm{ch}_k^{\alpha_0}(\CF_0)
\]
for any $\tilde{\alpha}$ of the type (\ref{alphatilde}). 


A direct calculation by applying the Grothendieck--Riemann--Roch formula to the natural projection 
\[
\mathrm{pr}: T^*C\times \CM_{\mathrm{Dol}} \to C\times \CM_{\mathrm{Dol}} 
\]
(see the paragraph before \cite[Remark 8]{Markman3}) together with Example \ref{example1} yields
\begin{equation*}\label{997}
\int_{\mathrm{sp}^!(\gamma)} \mathrm{ch}_k^{\alpha_0} (\CF_0) = \int_{j_C^*\gamma} \mathrm{ch}_{k-1}^{\alpha_0}(\CU).
\end{equation*}
Here $\CU = \mathrm{pr}_* \CF_0$ is a universal family on $C\times \CM_{\mathrm{Dol}}$.

In conclusion, we obtain that
\[
\mathrm{sp}^!\left( \int_{\gamma} \mathrm{ch}_k^\alpha (\CF_\beta) \right) = \int_{j_C^*\gamma} \mathrm{ch}_{k-1}^{\alpha_0}(\CU).
\]
Moreover, we know from Proposition \ref{prop4.4} that the twisted class \[
\int_{j_C^*\gamma} \mathrm{ch}_{k-1}^{\alpha_0}(\CU),\quad \forall \gamma \in H^*(A, \BQ),~ \forall k\geq 1
\]
has perversity $k-1$. Hence Lemma \ref{lem4.7} and Remark \ref{rmk4.8} imply, by recalling (\ref{taut_class}), that 
\[
\int_{j_C^*\gamma} \mathrm{ch}_{k-1}^{\alpha_0}(\CU) = c(k-1, j_C^*\gamma). \qedhere
\]
\end{proof}

\subsection{Proofs of Theorems \ref{genus_2}, \ref{P=W_even}, and \ref{P=W_odd}}\label{Section4.6}

Since the perverse filtration with respect to the Hitchin fibration $h: \CM_{\mathrm{Dol}} \to \Lambda$ is locally constant when we deform the curve $C$ \cite{dCM}, it suffices to prove all the three theorems for a curve which can be embedded in an abelian surface
\[
j_C: C\hookrightarrow A.
\]
After having done so,  we can apply the specialization morphism 
\[
\mathrm{sp}^!: H^*(\CM_{\beta, A}, \BC) \to H^*(\CM_{\mathrm{Dol}}, \BC)
\]
introduced in Section \ref{section4.3}. The following results are immediate consequences of  Theorem \ref{strengthened2.1}, Proposition \ref{prop4.4}, and Proposition \ref{prop4.8}:
\begin{enumerate}
    \item[(i)] We have
    \begin{equation}\label{333333}
    c(\gamma, k) \in \widetilde{G}_kH^*(\CM_{\mathrm{Dol}}, \BQ), \quad \forall \gamma \in \mathrm{Im}({j_C^*}) \subset H^*(C, \BQ).
    \end{equation}
    \item[(ii)] The restriction of the decomposition \[
    H^*(\CM_{\mathrm{Dol}}, \BQ)= \bigoplus_{k,d}\widetilde{G}_kH^d(\CM_{\mathrm{Dol}}, \BQ)
    \]to the subalgebra of $H^*(\CM_{\mathrm{Dol}}, \BQ)$ generated by the classes (\ref{333333}) is multiplicative.
\end{enumerate}

We first prove Theorem \ref{genus_2}. The Abel-Jacobi morphism embeds a genus 2 curve $C$ into its Jacobian
\[
j_C: C \hookrightarrow \mathrm{Jac}(C) = A.
\]
Hence the restriction morphism $j_C^*$ is surjective, and Theorem \ref{genus_2} follows from (i) and (ii) above. 

The proof of Theorem \ref{P=W_even} is similar. For any embedding $j_C$, the image of $j_C^*$ always contains the sub-vector space
\[
H^0(C, \BQ) \oplus H^2(C, \BQ) \subset H^*(C, \BQ).
\]
Hence the subalgebra
\[
R^*(\CM_{\mathrm{Dol}}) \subset H^*(\CM_{\mathrm{Dol}}, \BQ)
\]
is contained in the subalgebra generated by the classes (\ref{333333}), and we again conclude Theorem \ref{P=W_even} by (i) and (ii).

Finally we treat the odd classes 
\[
c(\gamma, k) \in H^{2k-1}(\CM_{\mathrm{Dol}}, \BQ), \quad \gamma\in H^1(C, \BQ)
\]
and prove Theorem \ref{P=W_odd}. 

When the curve $C$ has genus $\geq 3$, the restriction
\begin{equation}\label{resH1}
j_C^*: H^1(A, \BQ) \to H^1(C, \BQ)
\end{equation}
is not surjective. We know from (i) that $c(\gamma, k)$ has perversity $k$ for any $\gamma$ lying in the image of (\ref{resH1}). Since the monodromy group of the moduli space $\CM_g$ of nonsingular genus g curves is the full symplectic group $\mathrm{Sp}_{2g}$ by \cite{A}, the sub-vector space 
\[
\mathrm{Im}\left({j_C^*}:H^1(A, \BQ)\to H^1(C,\BQ) \right) \subset H^1(C, \BQ)
\]
generates the total cohomology $H^1(C, \BQ)$ via the action of the monodromy group. We deduce Theorem \ref{P=W_odd} from \cite{dCM}. \qed

\end{document}